\documentclass[12pt]{article}
\usepackage{amsfonts, amsmath, amssymb, amsgen, amsthm, amscd,  
latexsym}
\usepackage{color}
\usepackage[all]{xy}

\def\Diff{\mathop{\rm Diff}\nolimits}

\def\GL{\mathop{\rm GL}\nolimits}

\def\Id{\mathop{\rm Id}\nolimits}

\def\Im{\mathop{\rm Im}\nolimits}

\def\ad{\mathop{\rm ad}\nolimits}
\def\det{\mathop{\rm det}\nolimits}

\def\log{\mathop{\rm log}\nolimits}

\def\Hom{\mathop{\rm Hom}\nolimits}
\def\exp{\mathop{\rm Exp}\nolimits}

\def\Cb{{\mathbb C}}

\def\Nb{{\mathbb N}}
\def\Rb{{\mathbb R}}

\def\Zb{{\mathbb Z}}

\def\Ac{{\cal A}}
\def\Bc{{\cal B}}

\def\Fc{{\cal F}}
\def\Gc{{\cal G}}
\def\Hc{{\cal H}}

\def\Ic{{\cal I}}
\def\Lc{{\cal L}}

\def\Pc{{\cal P}}

\def\Uc{{\cal U}}

\def\Qc{{\cal Q}}
\def\Kc{{\cal K}}
\def\Lc{{\cal L}}
\def\Cc{{\cal C}}

\def\Dc{{\cal D}}
\def\uc{{\cal u}}

\def\a{\alpha}
\def\b{\beta}
\def\d{\delta}
\def\D{\Delta}
\def\g{\gamma}

\def\G{\Gamma}

\def\om{\omega}
\def\Om{\Omega}
\def\s{\sigma}

\def\t{\theta}

\def\ve{\varepsilon}
\def\vp{\varphi}

\def\fl{\forall}
\def\ify{\infty}

\def\nb{\nabla}
\def\ot{\otimes}

\def\ra{\rightarrow}

\def\rt{\triangleright}

\def\lt{\triangleleft}
\def\cl{\blacktriangleright\hspace{-4pt} < }

\def\acl{\blacktriangleright\hspace{-4pt}\vartriangleleft }

\def\p{\partial}
\def\dbar{{^-\hspace{-7pt}\d}}
\def\gbar{{-\hspace{-9pt}\gamma}}
\def\etabar{{-\hspace{-7pt}\eta}}
\def\0D{\Delta^{(0)}}
\def\1D{\Delta^{(1)}}
\def\Db{\blacktriangledown}

\def\ts{\times}
\def\wg{\wedge}
\def\wt{\widetilde}
\def\td{\tilde}
\def\cop{{^{\rm cop}}}

\newcommand{\Fa}{\mathfrak{a}}

\newcommand{\Fg}{\mathfrak{g}}
\newcommand{\Fh}{\mathfrak{h}}

\newcommand{\Fl}{\mathfrak{l}}

\newcommand{\ybo}{{\bf y}}

\def\Tbo{{\bf T}}
\def\abo{{\bf a}}

\newtheorem{theorem}{Theorem}[section]
\newtheorem{remark}[theorem]{Remark}
\newtheorem{proposition}[theorem]{Proposition}
\newtheorem{lemma}[theorem]{Lemma}
\newtheorem{corollary}[theorem]{Corollary}

\newtheorem{definition}[theorem]{Definition}

\def\build#1_#2^#3{\mathrel{
\mathop{\kern 0pt#1}\limits_{#2}^{#3}}}
\newcommand{\ps}[1]{~\hspace{-4pt}_{^{(#1)}}}
\newcommand{\ns}[1]{~\hspace{-4pt}_{_{{<#1>}}}}

\newcommand{\lu}[1]{\;^{#1}\hspace{-2pt}}

\def\odots{\ot\cdots\ot}
\def\wdots{\wedge\dots\wedge}

\def\one{{\bf 1}}

\numberwithin{equation}{section}

 \newcommand{\ie}{{\it i.e.\/}\ }
 \def\cf{{\it cf.\/}\ }

\def\a{\alpha}
\def\b{\beta}

\def\d{\delta}

\def\g{\gamma}

\def\i{\iota}

\def\om{\omega}
\def\s{\sigma}
\def\t{\theta}
\def\ve{\varepsilon}
\def\ph{\phi}
\def\vp{\varphi}

\def\D{\Delta}
\def\G{\Gamma}

\def\Om{\Omega}

\def\Ph{\Phi}

\def\dt{\left.\frac{d}{dt}\right|_{_{t=0}}}

\newcommand{\at}[1]{\hspace{-4pt}\left.\frac{}{}\right|_{_{#1=0}}}

\def\fl{\forall}
\def\ify{\infty}

\def\nb{\nabla}

\def\ot{\otimes}
\def\part{\partial}

\def\ts{\times}
\def\wdg{\wedge}

\def\ra{\rightarrow}

\def\text{\hbox}

\def\fl{\forall}
\def\ify{\infty}

\def\nb{\nabla}
\def\ot{\otimes}

\def\ra{\rightarrow}

\def\ts{\times}
\def\wdg{\wedge}
\def\wt{\widetilde}

\def\Diff{\mathop{\rm Diff}\nolimits}

\def\Hom{\mathop{\rm Hom}\nolimits}

\def\GL{\mathop{\rm GL}\nolimits}

\def\Id{\mathop{\rm Id}\nolimits}
\def\exp{\mathop{\rm exp}\nolimits}

\def\Trace{\mathop{\rm Trace}\nolimits}

\def\build#1_#2^#3{\mathrel{
\mathop{\kern 0pt#1}\limits_{#2}^{#3}}}

\numberwithin{equation}{section}
\newcommand{\comment}[1]{\relax}

\def\tb{{\bf t}}
\def\Gb{{\bf G}}

\def\Nbo{{\bf N}}

\begin{document}
\title{Hopf algebras and universal Chern classes}

\author{
\begin{tabular}{cc}
Henri Moscovici \thanks{Research
    supported in part by the National Science Foundation
    award DMS-1300548
    and by a grant from the Romanian National Authority for Scientific Research,
project no. PN-II-ID-PCE-2012-4-0201.}~~\thanks{Department of Mathematics, 
The Ohio State University,
    Columbus, OH 43210, USA }\quad and \quad  Bahram Rangipour
    \thanks {Department of Mathematics  and   Statistics,
     University of New Brunswick, Fredericton, NB, Canada}
      \end{tabular}}


\maketitle

\begin{abstract}
We construct a variant $\Kc_n$ of the Hopf algebra $\Hc_n$, which
acts directly on the noncommutative model for the space of leaves 
of a general foliation rather
than on its frame bundle. We prove that the Hopf cyclic cohomology 
of $\Kc_n$ is isomorphic to that of the pair  $(\Hc_n, \Fg \Fl_n)$ and thus
consists of the universal Hopf cyclic Chern classes. We also realize these classes in terms of
geometric cocycles.
\end{abstract}

\section*{Introduction}

The application of Connes' cyclic category~\cite{Cext} to the cohomology
of Hopf algebras, originally employed to compute the local index formula~\cite{cm1} 
for hypoelliptic operators on spaces of leaves of foliations~\cite{cm2}, has 
stimulated the interest in developing a theory of Hopf cyclic characteristic 
classes in the framework of noncommutative geometry. To this end the 
geometric characteristic classes of foliations (see e.g.~\cite{BottHaef}) have been
gradually reconfigured in the context of Hopf cyclic cohomology~\cite{mr09, mr11, RS4, M-GC, M-EC}, which holds the potential of being applicable to other
noncommutative spaces (\cf~\cite{CM04}). 
\smallskip

In this paper we construct a variant $\Kc_n$ of the Hopf algebra $\Hc_n$  
(\cf~\cite{cm2, mr09}), which 
acts directly on the noncommutative model for the generic space of leaves rather
than on its frame bundle. Associated to a Kac decomposition of the group
$\Diff(\Rb^n)$ distinct from that employed in defining $\Hc_n$, 
the Hopf algebra $\Kc_n$ has
a different Hopf cyclic cohomology, which is
no longer identifiable as the Gelfand-Fuks
cohomology of the Lie algebra $\Fa_n$ of formal vector fields of $\Rb^n$.
Instead, in analogy with the van Est isomorphism
for algebraic groups (see~\cite{Kumar}), the (absolute) Hopf cyclic cohomology 
of $\Kc_n$ is canonically isomorphic to the relative Lie algebra cohomology of
the pair $(\Fa_n, \Fg \Fl_n)$, or equivalently to the
Hopf cyclic cohomology of the pair $(\Hc_n, \Fg \Fl_n)$ (\cf~\cite{mr09, mr11}), 
and therefore it also is a repository of the universal Hopf cyclic 
Chern classes. The proof of this isomorphism is achieved by
supplementing our earlier techniques with those in~\cite{RS}. By a 
construction parallel to that in~\cite{M-GC}, we then realize these classes in terms of
concrete geometric cocycles, in the spirit of the Chern-Weil theory.
\smallskip

The paper is organized as follows. In \S \ref{S1} we define the Hopf algebra $\Kc_n$
via its natural action on the \'etale groupoid $\Rb^n \rtimes \Diff(\Rb^n)^\d$. In broad
outline this construction parallels that of $\Hc_n$ at
the level of the prolongation groupoid on frame bundle.
However, unlike $\Hc_n$, $\Kc_n$ is no longer isomorphic as an algebra with
a quotient of a universal enveloping algebra, and its antipode is more intricate. 
To prove that it actually is
a Hopf algebra we employ the Lie-Hopf algebra techniques developed in~\cite{RS}
in order to realize it as a bicrossed product.
\smallskip

In \S \ref{S2} we refine the Lie-Hopf algebra decomposition of $\Kc_n$ by means
of a further bicrossed product factorization of its commutative Hopf subalgebra $\Fc_\Kc$. 
Using the full factorization so obtained, we then prove in  \S \ref{S3}
that the Hopf cyclic cohomology of $\Kc_n$ is isomorphic to the relative
 Hopf cyclic cohomology of the pair $(\Hc_n, \Fg \Fl_n)$, and therefore (\cf \cite{mr09})
to the truncated polynomial  ring of Chern classes. 
A crucial ingredient of the proof is supplied by~\cite[Theorem 4.10]{RS}, which provides
the appropriate version of the van Est isomorphism in the present context.
\smallskip

The Chern-Weil type construction of cocycles representing the Hopf cyclic classes of $\Kc_n$ is 
carried out in \S \ref{S4}, by an adaptation of the methods developed in~\cite{M-GC, M-EC}.
This construction is then illustrated in a very concrete fashion in \S \ref{S5}, where we
produce completely explicit cocycles representing the Hopf cyclic classes of  $\Kc_1$ 
and of the pair $(\Hc_1, \Fg \Fl_1)$. The detailed calculation shows clearly how the
equivariant Chern classes $c_0 \equiv 1$ and $c_1$ in the Bott complex become
Hopf cyclic Chern classes of  $\Kc_1$, respectively $(\Hc_1, \Fg \Fl_1)$. Particularly
noteworthy is the metamorphosis of the (secondary) Godbillon-Vey cocycle of $\Hc_1$ into a
representative of a (primary) Chern class. To further clarify this phenomenon we follow through
with an additional calculation which shows how to restore the Godbillon-Vey class.
This also serves to point out that in order to incorporate the secondary Hopf cyclic transverse
characteristic classes at the same direct level it is necessary to pass 
to a topological enhancement of the Hopf algebra $\Kc_1$.
\smallskip 

Consisting exclusively of primary classes, the Hopf cyclic cohomology of 
$\Kc_n$ is perfectly positioned to be a receptacle for the yet elusive
transverse index formula representing the Connes-Chern character in equivariant
$K$-homology. On the other hand, as mentioned above,
a full Hopf cyclic representation of all the geometric 
transverse characteristic classes requires the passage to a 
topological enhancement of these Hopf algebras. This is a separate development of
interest in its own, which
is work in progress and will make the object of a forthcoming paper~\cite{mr16}.


\tableofcontents

\section{The Hopf algebra $\Kc_n$ and its standard action}\label{S1}

Let $M=\Rb^n$ and let $\Gb:=\Diff(M)^\d$ be the group of all orientation preserving 
diffeomorphisms of $M$ equipped with the discrete topology.  
The Hopf algebra $\Kc_n$ arises in the same way as $\Hc_n$ (\cf \cite{cm2}),
only at the level of the action groupoid $M \ltimes \Gb$ rather than of its frame bundle 
prolongation $M \ltimes \Gb$.

Thus, we consider the crossed product algebra 
$\Ac \equiv \Ac_\Gb:= C^{\infty}_c(M)\rtimes \Gb$, 
where $\Gb$ acts on $C^{\infty}_c(M)$ by $\vp\rt f=f\circ\vp^{-1}$.  
A typical element of $\Ac_\Gb$ is a finite sum $\sum_if_iU^\ast_{\vp_i}$, where $f_i\in C^\infty_c(M)$ and $U^\ast_{\vp_i}$ stands for $\vp_i^{-1}\in\Gb$.  The product of $\Ac_\Gb$ is determined by the multiplication rule
\begin{equation}
fU^\ast_\vp \;gU^\ast_\psi= f \cdot (g\circ\vp) U^\ast_{\psi\vp}.
\end{equation}

The vector fields $X_k \cong \p_k:=\frac{\p}{\p x^k}$ is made to act on $\Ac_\Gb$ by
\begin{equation}
X_k(fU^\ast_\vp)= X_k(f)U^\ast_\vp= \p_k(f) U^\ast_\vp.
\end{equation}
One observes that
\begin{align}\label{X-M-A}
\begin{split}
X_k(fU^\ast_\vp \cdot gU^\ast_\psi)&= X_k(f\cdot (g\circ \vp)) U^\ast_{\psi\vp}\\
 &=\p_k(f)U^\ast_{\vp}\cdot g U^\ast_{\psi} +  
 \sum_i \p_k(\vp^i)\cdot f U^\ast_{\vp}\;\p_i(g) U^\ast_{\psi},
\end{split}
\end{align}
which proves that for all $a,b\in \Ac$ one has
\begin{equation} \label{Leib1}
X_k(ab)=X_k(a)b+\sum_i\s^i_k(a)X_i(b),
\end{equation}
where  
\begin{equation}  \label{Leib2}
\s^i_j(fU^\ast_\vp)=\p_j(\vp^i)fU^\ast_\vp.
\end{equation}
Another elementary manipulation  gives
\begin{align}\label{s-1-M-A}
\s^i_j(fU^\ast_\vp\cdot gU^\ast_\psi) = \s^i_j(f\cdot (g\circ \vp)) U^\ast_{\psi\vp}=  
\sum_k   \p_j(\vp^k)\, f U^\ast_{\vp} \cdot \p_k(\psi^i)\, g U^\ast_\psi ,
\end{align}
showing that for any $a,\;b\in \Ac$
\begin{equation}
\s^i_j(ab)=\sum_k \s^i_k(a) \; \s^k_j(b).
\end{equation}

In addition, we introduce the Jacobian operator,
\begin{equation}
\s(fU^\ast_\vp)= \det J(\phi)\cdot fU^\ast_\vp,
\end{equation}
where $J(\phi)(x) = \phi' (x)$ stands for the Jacobian  matrix of $\phi$ at $x \in V$.
It is an algebra automorphism of $\Ac$, whose inverse acts by
\begin{equation}
\s^{-1}(fU^\ast_\vp)=\det J(\vp^{-1})\circ\vp\; fU^\ast_\vp= \det J(\vp)^{-1} \, f U^\ast_\vp.
\end{equation}
 
\begin{definition}
The algebra $\Kc_n$ is the unital subalgebra of  $\Lc(\Ac_\Gb)$ generated by the
operators $\{X_k , \, \s^i_j , \, \s^{-p} \,\mid \, i,j,k=1, \ldots , n ; \, \, p\in\Nb\}$.
\end{definition}

Note that the algebra $\Kc_n$ also contains 
$\s =  \sum_{\pi \in S_n}(-1)^\pi \s^1_{\pi(1)} \cdots \s^n_{\pi(n)}$, as well as
the operators 
$\s^i_{j_1,\ldots,j_k}$, where $1\le i,  j_1, \ldots,j_k\le n$,
\begin{equation}
\s^i_{j_1,\ldots,j_k}(fU^\ast_\vp)=\p_{j_k}\cdots\p_{j_1}(\vp^i)\cdot fU^\ast_\vp ,
\end{equation}
which are iteratively generated by the commutators
\begin{equation}
[X_\ell\;,\, \s^{i}_{j_1,\ldots,j_k}]= \s^{i}_{j_1,\ldots,j_k,l},\\\label{bracket-X-X} 
\end{equation}
Other obvious relations are
\begin{align}
&[X_i,X_j]=0 , \\ 
&\s^{i}_{j_1,\ldots,j_k}=\s^{i}_{j_{\pi(1)},\ldots,j_{\pi(k)}}, \quad \text{for any permutation }
\pi \in S_k, \\\label{brackets-sigma-sigma}
&[\s^{i}_{j_1,\ldots,j_k}\,,\,\s^{p}_{q_1,\ldots,q_m} ]=0,
\qquad [\s^{-1} , \s^{i}_{j_1,\ldots,j_k}]=0, \\  \label{rel-sigma-det}
&\s^{-1} \,  \sum_{\pi \in S_n}(-1)^\pi \s^1_{\pi(1)} \cdots \s^n_{\pi(n)} \, = 1 \, , \\
&   [X_k, \s^{-1} ] = -  \s^{-2} \, 
\sum_{\pi \in S_n} (-1)^\pi \bigl(\s^1_{\pi(1) k} \cdots \s^n_{\pi(n)} + \ldots \\ \notag
&\qquad \qquad \ldots + \s^1_{\pi(1)} \cdots \s^n_{\pi(n) k}\bigr) .
\end{align}

 \begin{proposition}
The following collection of operators forms linear basis for $\Kc_n$ :
\begin{align}\label{PBW}
\s^{-p}\s^{i_1}_{J_1}\cdots \s^{i_m}_{J_{m}}X_{\ell_1}^{q_1}\cdots X_{\ell_k}^{q_k} \, ;
\end{align}
here $1\le i_1, \ldots , i_m \le n$, $1\le \ell_1, \ldots ,\ell_k \le n$, 
$p, \, q_1, \ldots ,q_\ell \in \Zb^+$, and $J_p$ are finite ordered sets $J_p=\{j_1\le j_2\le\cdots\le j_{m_p}\}$, with $1\le j_r\le n$.
\end{proposition}
\proof Similar to that of~\cite[Proposition 1.3]{mr09} .
\endproof
\smallskip

We use the action of $\Kc_n$ on $\Ac_\Gb$ and the corresponding Leibnitz rules, such as
 \eqref{Leib1},  \eqref{Leib2}, to equip $\Kc_n$
with the bialgebra structure defined by the condition
\begin{align}
&\D(k)=k\ps{1}\ot k\ps{2} \qquad \text{iff}\qquad k(ab)=k\ps{1}(a)k\ps{2}(b) ,\\
&\ve(k)1_\Ac=k(1).
\end{align}
In particular,
\begin{align}\label{D-X-l}
&\D(X_\ell)=X_\ell\ot 1+ \s^k_\ell\ot X_k,\\\label{D-s-i-j}
&\D(\s^i_j)=\s^k_j\ot \s^i_k,\\
&\D(\s)=\s\ot \s, \quad \D(\s^{-1})=\s^{-1}\ot \s^{-1},\\  \label{D-s-i-j-k}
&\D(\s^i_{j,k})=\s^m_{j,k}\ot \s^i_m+ \s^\ell_k\s^m_j\ot \s^i_{m,l} , \\
&\D(\s^i_{j_1,\ldots,j_k})=[\D(X_{j_k}), \D(\s^i_{j_1,\ldots,j_{k-1}})], \\
&\ve(\s)=\ve(\s^{-1})=1,\qquad \ve(\s^i_j)=\d^i_j, \qquad \ve(X_\ell)=\ve(\s^i_{j_1,\ldots,j_k})=0.
\end{align}


Let $\Kc_{\rm  ab}$ be the commutative polynomial algebra generated by $\s^i_{j_1, \ldots, j_k}$ and $\s^{-1}$. It is obvious that $\Kc_{\rm  ab}$ is a subbialgebra of $\Kc_n$.

For  $k\in \Kc_{\rm  ab}$ we define $\gbar_\Kc(k):\Gb \ra C^{\infty}(V)$ by
\begin{equation} 
\gbar_\Kc(k)(\psi)=k(U^\ast_\psi)U_\psi.
\end{equation}
One readily checks that the following cocycle property holds
\begin{equation}\label{cocycle-property}
\gbar_\Kc(k)(\psi_1\psi_2)=\gbar_\Kc(k\ps{1})(\psi_2)\gbar_\Kc(k\ps{2})(\psi_1)\circ\psi_2.
\end{equation}
Using the above operators
we then define the map $S:\Kc_{\rm  ab}\ra \Kc_{\rm  ab}$ by
\begin{equation}\label{def-S-F}
S(f)(g U^\ast_\psi)=\gbar_\Kc(f)(\psi^{-1})\circ\psi\cdot   g\;U^\ast_{\psi}.
\end{equation}

\begin{lemma}
The map $S$ defined in \eqref{def-S-F} is the antipode of $\Kc_{ \rm  ab}$ and hence $\Kc_{\rm  ab}$ is a Hopf algebra.
\end{lemma}

\begin{proof}
We should show that $S$ is the inverse of $\Id_{\Kc_{\rm  ab}}$ in the convolution algebra $\Hom(\Kc_{\rm  ab},\Kc_{\rm ab})$.
Indeed we first verify that $S$ is the left inverse,
\begin{align}
\begin{split}
&(\Id_{\Kc_{\rm  ab}}\ast S)(f)(U^\ast_\psi)= f\ps{1} S(f\ps{2})(U^\ast_\psi)= f\ps{1}(\gbar_\Kc(f\ps{2})(\psi^{-1})\circ{\psi}\; U^\ast_{\psi})\\
&=\gbar_\Kc(f\ps{1})(\psi)\gbar_\Kc(f\ps{2})(\psi^{-1})\circ{\psi}\; U^\ast_{\psi}=\gbar_\Kc(f)(e) U^\ast_\psi= \ve(f)U^\ast_\psi.
\end{split}
\end{align}
Here in the last two equalities we have used the cocycle property \eqref{cocycle-property} of $\gbar_\Kc$ and the very definition of $\ve$. Similarly, one proves that $S$ is a
right convolution inverse to $\Id_{\Kc_{\rm  ab}}$.
\end{proof}

To equip the algebra $\Kc_n$ itself with a Hopf algebra structure, 
we shall check that the Lie algebra $V$ generated by the $X_\ell$'s
together with the copposite Hopf algebra ${\Fc_\Kc} ={\Kc_{\rm ab}}^{\rm cop}$  
form a Lie-Hopf pair in the sense of~\cite{RS}. It will follow that
the universal enveloping algebra $\Uc (V)$ of $V$ together with  ${\Fc_\Kc}$
form a matched pair of Hopf algebras (\cf \cite{maj}), from which we will
reassemble $\Kc_n$.

The Lie algebra $V$  acts on  ${\Fc_\Kc}$ from the  left  via
\begin{equation}\label{action-g-F-cop}
\rt: V\ot {\Fc_\Kc}\ra {\Fc_\Kc},\qquad X\rt f=[X,f].
\end{equation}
Explicitly, for $f \in {\Fc_\Kc}$,
\begin{align}\label{action-vs-bracket}
\begin{split}
&(X_\ell\rt f)(g U^\ast_\psi)= (X_\ell f-fX_\ell)(gU^\ast_\psi) 
= X_\ell( \gbar_\Kc(f)g U^\ast_\psi) -f(\p_\ell(g)U^\ast_\psi)\\
 &=\p_\ell(\gbar_\Kc(f)g U^\ast_\psi +\gbar_\Kc(f)\p_\ell(g) U^\ast_\psi - \gbar_\Kc(f)(\psi)\p_\ell(g)U^\ast_\psi \\
 & = X_\ell(\gbar_\Kc(f)) gU^\ast_\psi.
\end{split}
\end{align}

We also define the following right coaction of $\Fc_\Kc$ on $V$ by
\begin{equation}\label{coaction-g-F-cop}
\Db: V\ra  V\ot {\Fc_\Kc},\Db(X_\ell)= X_k\ot \s^k_\ell\;.
\end{equation}

\begin{lemma}
Via the action and coaction defined in \eqref{action-g-F-cop} and  \eqref{coaction-g-F-cop}, ${\Fc_\Kc}$ is a $ {V}$-Hopf algebra.
\end{lemma}
\begin{proof}
Using the coproduct of $\s^i_j$, it is straightforward to see that \eqref{coaction-g-F-cop} defines a coaction.
We need to verify that the action $\rt$ and the coaction $\Db$ satisfy the conditions required
for a Lie-Hopf pair (\cf~\cite{RS}). 

 First we should check that for any $g\in {\Fc_\Kc}$ and any $X\in  {V}$ one has
\begin{equation}
\D(X\rt g)= X\bullet \D(g)=  g\ps{1}\ot X\rt g\ps{2} + X\ns{0}\rt g\ps{1}\ot X\ns{1} g\ps{2}.
\end{equation}
Indeed if $a,b\in \Ac$  and $f\in \Kc_{\rm ab}$ then
\begin{align}\
\begin{split}
&\D(X_\ell\rt f)(ab)=[X_\ell,f](ab)= X_\ell f(ab)-fX_\ell(ab)=\\
&X_\ell(f\ps{1}(a)f\ps{2}(b))- f(X_\ell (a)b+ \s^k_\ell(a) X_k(b))=\\
&X_\ell(f\ps{1}(a))f\ps{2}(b)+ \s^k_\ell(f\ps{1}(a))X_k(f\ps{2}(b)) -\\
 &f\ps{1}(X_\ell (a))f\ps{2}(b)- f\ps{1}(\s^k_\ell(a)) f\ps{2}(X_k(b))=\\
&  [X_\ell,f\ps{1}](a) f\ps{2}(b) + \s^k_\ell f\ps{1}(a) [X_k,f\ps{2}](b)=\\
&X_\ell\rt f\ps{1}(a)f\ps{2}(b)+ (X_\ell\ns{1}f\ps{1})(a) (X_\ell\ns{0}\rt f\ps{2})(b).
\end{split}
\end{align}
Thus, for any $X\in  {V}$,
\begin{equation}
\D(X\rt f)= X\rt f\ps{1}\ot f\ps{2}+ X\ns{1}f\ps{1}\ot X\ns{0}\rt f\ps{2}.
\end{equation}
Since the Lie algebra $ {V}$ is commutative and $\s^i_{j,k}=\s^{i}_{k,j}$ the coaction  $\Db$ satisfies the structure identity of $ {V}.$ 

Finally, $\ve(X_\ell\rt f)=0$ for any $f\in \Kc_{\rm ab}$, which completes the 
verification of the axioms of a Lie-Hopf pair.
\end{proof}

As a consequence, the bicrossed product Hopf algebra ${\Fc_\Kc}\acl\Uc(V)$ is well-defined.
We define the map
\begin{equation}
\Ic: {\Fc_\Kc}\acl\Uc(V)\ra \Kc_n^{\rm cop},\qquad \Ic(f\acl u)\;=\;fu.
\end{equation}

\begin{proposition}
The above map $\Ic$ is an isomorphism of bialgebras.
\end{proposition}
\begin{proof}
One uses the linear basis \eqref{PBW} for $\Kc$ to see that $\Ic$ is an isomorphism of vector spaces.
Let us  check that the map $\Ic$ is an algebra homomorphism.
We first use \eqref{action-vs-bracket} to see that $uf= u\ps{1}\rt f u\ps{2}$ in $\Kc_n$,
\begin{align}
\begin{split}
&uf(g U^\ast_\psi)= u(\gbar_\Kc(f)(\psi)gU^\ast_\psi)= u(\gbar_\Kc(f)(\psi)g)U^\ast_\psi=\\
&u\ps{1}(\gbar_\Kc(f)(\psi)) u\ps{2}(g)U^\ast_\psi= ((u\ps{1}\rt f) u\ps{2})(gU^\ast_\psi).
\end{split}
\end{align}
This shows that $\Ic$ is indeed an algebra map.

To show that $\Ic$ is a coalgebra map it is necessary and  sufficient 
 to  check that $\Ic$ commutes with coproduct on the generators. Indeed,
\begin{align*}
\begin{split}
&(\Ic\ot \Ic)(\D_{{\Fc_\Kc}\acl\Uc(V)}(f\acl 1))=(\Ic\ot \Ic)(f\ps{2}\acl 1)\ot \Ic(f\ps{1}\acl 1)\\
& =f\ps{2}\ot f\ps{1}=\Db_{\Kc_{\rm cop}}(f) ; \\
&(\Ic\ot \Ic)(\D_{{\Fc_\Kc}\acl\Uc(V)}(1\acl X_\ell))=(\Ic\ot\Ic)(1\acl X_k\ot \s^k_\ell\acl 1+ 1\acl 1\ot 1\acl X_\ell)\\
&  =  X_k \ot \s^k_\ell + 1\ot X_\ell=\Db_{\Kc_{\rm cop}}(X_\ell).
\end{split}
\end{align*}
\end{proof}

\begin{corollary}
The following defines the antipode of  $\Kc_n$:
 \begin{equation}
S_\Kc(u)= S_{\Fc_\Kc}(u\ns{1})S_\Uc(u\ns{0}), \quad S_\Kc(f)=S_{\Fc_\Kc}(f).
\end{equation}
Hence $\Kc_n$ is a Hopf algebra and $\Ic$ is an isomorphism of Hopf algebras.
\end{corollary}

\begin{proof}
One uses the antipode definition for a bicrossed product
\begin{equation}\label{anti}
S(f\acl u)=(1\acl S(u\ns{0}))(S(fu\ns{1})\acl 1) , \qquad f \in \Fc , \, u \in \Uc.
\end{equation}
and the fact that $\Uc$ is cocommutative and ${\Fc_\Kc}$ is commutative to see
\begin{equation}
S^{-1}_{{\Fc_\Kc}\acl \Uc}(f\acl u)= S_{\Fc_\Kc}(fu\ns{1})\acl  S_\Uc(u\ns{0}).
\end{equation}
Since $\Ic$ is isomorphism of bialgebras and ${\Fc_\Kc}\acl \Uc$ is a Hopf algebra,  $\Ic$ induces a unique  antipode  on $\Kc_n$. Equivalently,
 $S:\Kc\ra \Kc$ is defined by  the identity
 $$S_\Kc= \Ic\circ {S^{-1}_{{\Fc_\Kc}\acl \Uc}}\circ\Ic^{-1}.$$
\end{proof}

One sees that
\begin{align}
S(\s^i_j)= \s^{-1} m^j_i  ,
\end{align}
where $m^p_q= (-1)^{p+q}\det M^p_q$, with $M^p_q$ signifying
the $(n-1)\times( n-1)$ matrix obtained by removing the $q$th row and $p$th 
column of the matrix $[\s^i_j]$. Also,
\begin{align}\label{S-K}
\begin{split}
&S(\s)=\s^{-1},\quad S(\s^{-1})=\s, \\
&S(\s^i_{jk})= -S(\s^i_r)S(\s^s_j)S(\s^t_k)\s^r_{st},\\
&S(X_k)=-S(\s^\ell_k)X_\ell.
\end{split}
\end{align}


\section{Bicrossed product decomposition of $\Fc_\Kc$} \label{S2}
\label{S2}
   We start  with the decomposition of $\Gb= \Tbo\cdot \Gb^\dagger$, where
\begin{align}\label{decopmpose-G-}
&\Tbo=\left\{\vp\in \Gb\mid \vp(x)=x+b,\;\; \text{for some }\;\; b\in \Rb^n \right\},\\\label{decopmpose-G+}
&\Gb^\dagger= \{\psi\in\Gb \mid \psi(0)=0\}.
\end{align}
Any $\phi\in\Gb$  can be written uniquely as
  \begin{equation} \label{phi-dec}
  \phi= \vp_\phi\circ\psi_\phi, \qquad \qquad  \vp_\phi\in \Tbo,\; \;\psi_\phi\in \Gb^\dagger.
  \end{equation}
   where
\begin{equation}
\vp_\phi(x)= x+\phi(0), \quad \psi_\phi(x)=\phi(x)-\phi(0).
\end{equation}
This yields that $(\Tbo,\Gb^\dagger)$ is a matched pair of groups with respect
to the left action of $\Gb^\dagger$ on $\Tbo$ and the right action of $\Tbo$ on $\Gb^\dagger$ 
determined by
\begin{equation}
\psi\circ\vp= (\psi\rt \vp)\circ(\psi\lt\vp).
\end{equation}
Thus, for  $\vp\in G^-$ defined by $\vp(x)=x+b$, and for $\psi\in G^+$ one has
\begin{equation}
\psi\rt \vp(x)=x+\psi(b), \qquad \psi\lt \vp(x)=\psi(x+b)-\psi(b).
\end{equation}
The first equation shows that under the canonical identification of $\Rb^n$ with 
the translation group, $\vp \in \Tbo \leftrightarrow b =\vp(0) \in \Rb^n$, 
the action of  $\Gb^\dagger$ on $\Tbo$ is just its natural action on $\Rb^n$. 
\smallskip

Let $\Fc(\Gb^\dagger)$ be the commutative unital algebra of functions on $\Gb^\dagger$
generated by the coefficients of the Taylor expansion at $0$,
\begin{align} \label{betas}
\begin{split}
&\b^i_{j_1,\ldots,j_k}(\psi)=\p_{j_k}\ldots \p_{j_1}\psi^i(x)\at{x}, \qquad 1\le i, j_1, \ldots, j_k\le n, \,\; \psi\in \Gb^\dagger,\\ &\b^{-1}(\psi)= \frac{1}{\det(\b^i_j\psi)}.
\end{split}
\end{align}

One proves as in~\cite[Proposition 2.5]{mr09} that the group  structure of $\Gb^\dagger$ induces a  Hopf algebra structure on $\Fc(\Gb^\dagger)$, determined by
\begin{align}
&\D(f)(\psi_1,\psi_2)= f(\psi_1\circ\psi_2), \qquad \psi_1,\psi_2\in \Gb^\dagger,\\
&S(f)(\psi)=f(\psi^{-1}), \qquad \psi\in \Gb^\dagger,\\
&\ve(f)=f(e).
\end{align}
One notes that for $\s^i_{j_1,\ldots,j_k}\in \Kc_n$ we have
\begin{equation}
\b^i_{j_1,\ldots,j_k}(\psi)=\gbar(\s^i_{j_1,\ldots,j_k})(\psi)(0), \qquad \b^{-1}(\psi)=\gbar(\s^{-1})(\psi)(0).
\end{equation}
There is a unique isomorphism of Hopf algebras
 \begin{equation}\i:{\Kc_{\rm ab}}^{\rm cop}\cong \Fc_\Kc\ra \Fc(\Gb^\dagger),
 \end{equation}
 with the property that
\begin{equation}
\i(\s^{-1})=\b^{-1}, \quad \i(\s^i_{j_1,\ldots,j_k})=\b^i_{j_1,\ldots,j_k}, \qquad 1\le i,j_1,\ldots,j_k\le n, \;\; k\in \Nb.
\end{equation}


One uses the right action of $\Tbo$ on $\Gb^\dagger$ to get a left action of $\Tbo$ on $\Fc(\Gb^\dagger)$ by
\begin{equation}\label{action-G-F}
\vp\rt f(\psi)= f(\psi\lt \vp).
\end{equation}
We identify $ {V}$ with the Lie algebra of the Lie group $\Tbo$, as the left invariant vector fields, and hence get the following action of $ {V}$ on $\Fc(\Tbo)$
\begin{equation}\label{action-V-F(G)}
(X\rt f)(\psi)= \dt f(\psi\lt  \exp(tX)), \qquad f\in \Fc(\Gb^\dagger),\;\psi\in \Gb^\dagger\; X\in  {V}.
\end{equation}

To illustrate this action on the generators we compute:
\begin{align}
\begin{split}
&\p_j(\psi\lt \exp(tX_\ell))^i(x)\at{x}= \\ &\p_j(\psi^i(x+te^\ell)-\psi^i(te^\ell))=(\p_j\psi^i)(x+te^\ell)\at{x}=(\p_j\psi^i)(te^\ell)
\end{split}
\end{align}
and continue by
\begin{align}
\begin{split}
&X_\ell\rt\b^i_j(\psi)= \dt \b^i_j(\psi\lt \exp(tX_\ell))=\dt (\p_j\psi^i)(te^\ell)=\\
&\p_\ell\p_j\psi^i(x)\at{x} =\b^i_{j,l}(\psi).
\end{split}
\end{align}
Similarly one proves that
\begin{equation}
X_\ell\rt \b^i_{j_1,\ldots, j_k}= \b^i_{j_1,\ldots, j_k,l}.
\end{equation}
One observe that the action of $\Gb^\dagger$ on $\Tbo$ is 
smooth and hence  induces an action of $\Gb^\dagger$ on ${V}$,
\begin{equation}
\psi\rt X(g)=\dt g(\psi\rt \exp(tX)), \qquad g\in C^{\infty}(\Rb^n).
\end{equation}

In dual fashion, the action of $\Gb^\dagger$ on ${V}$ defines a coaction
\begin{equation}\label{coaction-V}
\Db_V:V\ra {V}\ot \Fc(\Gb^\dagger),
\end{equation}
defined by
\begin{equation}\label{coaction-F-1-V}
\Db_V(X_\ell)= X_j\ot f^j_\ell, \qquad \text{if and only if}\;\; f^j_\ell(\psi)X_j= \psi\rt X_\ell
\end{equation}
Let us explicitly compute this coaction. 
Since the action of $\Gb^\dagger$ on $\Tbo$ is the natural one,
\begin{align}
&\psi\rt \exp (tX_\ell)(x)=x+ \psi(te^\ell).
\end{align}
So for any $g\in C_c^{\infty}(\Rb^n)$ one has,
\begin{align}
\begin{split}
&(\psi\rt X_\ell)(g)(x)=\dt g(\psi\rt \exp (tX_\ell))=\\
&\dt g(\psi(te^\ell))=(\p_k g)(x)\p_\ell\psi^k(x)\at{x}= X_k(g)(x)\b^i_j(\psi).
\end{split}
\end{align}
We thus proved that 
$$\Db_V(X_\ell)=X_k\ot \b^k_\ell .
$$
 This show that 
  \eqref{action-V-F(G)},  makes $\Fc(\Gb^\dagger)$ a $\Uc(V)$-module algebra and
the map $\i:\Fc_\Kc \ra \Fc(\Gb^\dagger)$ is ${V}$-linear, where action of ${V}$ on $\Fc_\Kc$ is defined by \eqref{action-g-F-cop}. Via the coaction defined by \eqref{coaction-F-1-V} $\Fc(\Gb^\dagger)$ is a ${V}$-Hopf algebra. The map $\i$ induces the following isomorphism
\begin{equation}
\i\acl\Id:\Fc_\Kc\acl \Uc(V)\ra \Fc(\Gb^\dagger)\acl \Uc(V),
\end{equation}

Now we decompose the group $\Gb^\dagger$ into $\Gb^\dagger_0\cdot \Nbo$, where
\begin{align}\label{decopmpose-G1+}
&\Gb^\dagger_0=\{\psi\in \Gb^\dagger\mid \; \psi(x)=ax,\quad a\in \GL_n\},\\\label{decopmpose-G2+}
&\Nbo=\{\psi\in \Gb^\dagger\mid\; \psi'(0)=\Id\}.
\end{align}
Precisely, for any $\psi\in c$ we define $\lambda \in \Gb^\dagger_0$, and 
$\nu\in \Nbo$ by
\begin{equation}
\lambda_\psi(x)=\psi'(0)x, \qquad \nu_\psi(x)= (\psi'(0))^{-1}\psi(x).
\end{equation}
This unique  decomposition determines the actions  of $\Gb^\dagger_0$ on $\Nbo$  and 
of $\Nbo$ on $\Gb^\dagger_0$, by the prescription
\begin{equation}\label{decomposition-G-1-G-2}
\nu\circ\lambda=(\nu\rt \lambda)\circ (\nu\lt \lambda) ,
\end{equation}
 for $\lambda\in \Gb^\dagger_0$ and $\nu\in \Nbo$. More exactly, with
 $\lambda (x) = \abo \cdot x$, \, $\abo \in \GL_n(\Rb)$, 
\begin{equation}
\nu\rt \lambda=\lambda, \quad \text{and} \quad (\nu\lt \lambda) (x)= \abo^{-1}\nu(\abo\cdot x) ,
\end{equation}
reflecting the fact that $\Nbo$ is a normal subgroup of  $\Gb^\dagger$.

We let $\Fc(\Gb^\dagger_0)$  be the algebra generated by the functions
\begin{align}
\begin{split}
\a^i_j(\lambda)= \p_j\lambda^j(x)\at{x} = a^i_j \, , \quad
\a^{-1}(\lambda)=\det(\abo) ,  \, \, \text{for}
\quad \lambda (x) = \abo \cdot x ,
\end{split}
\end{align}
\ie the algebra $\Pc(\GL_n)$ of regular functions on $\GL_n(\Rb)$.

Similarly, we let $\Fc(\Nbo)$ be the algebra generated by the restrictions to
$\Nbo$ of the Taylor coordinates \eqref{betas} on $\Gb^\dagger$,
\begin{align}\label{F-2-alpha}
\begin{split}
\a^i_{j_1, \ldots,j_k}(\nu):= \b^i_{j_1, \ldots,j_k}(\nu)= \p_{j_k}\cdots\p_{j_1}\psi(x)\at{x}
, \qquad  \nu\in \Nbo.
\end{split}
\end{align}
Once again, $\Fc(\Gb^\dagger_0)$ and $\Fc(\Nbo)$ are in fact Hopf algebras with the usual structure, 
\begin{align}
\D(f)(\psi_1,\psi_2)= f(\psi_1\circ\psi_2), \quad \ve(f)=f(e), \quad S(f)(\psi)=f(\psi^{-1}).
\end{align}
In particular
\begin{align}
&\D(\a^i_j)=\a^i_k\ot \a^k_j, \quad \D(\a^{-1})=\a^{-1}\ot \a^{-1},\\
&\D(\a^i_{j,k})=\a^i_{j,k}\ot 1+1\ot \a^i_{j,k}
\end{align}
Thus, the restriction maps of Hopf algebras
\begin{align}
\begin{split}
&\pi_1:\Fc(\Gb^\dagger)\ra \Fc(\Gb^\dagger_0), \qquad \pi_2:\Fc(\Gb^\dagger)\ra \Fc(\Nbo), \\
&\pi_1(\b^i_j)=\a^i_j,  \quad \pi_1(\b^{-1})=\a^{-1},\quad \pi_1(\b^i_{j_1,\ldots,j_k})=0,\\
&\pi_2(\b^i_j)=\d^i_j,\quad \pi_2(\b^{-1})=1,\quad \pi_2(\b^i_{j_1,\ldots,j_k})=\a^i_{j_1, \ldots,j_k}
\end{split}
\end{align}
are maps of Hopf algebras. These projections admit as cross-sections
the obvious inclusion maps $\Ic_i:\Fc(\Gb_i^+)\ra\Fc(\Gb^\dagger),$ $i=1,2$.
\begin{equation}
   \Ic_1(\a^i_j)=\b^i_j, \quad \Ic_1(\a^{-1})=\a^{-1}, \quad \Ic_2(\a^i_{j_1,\ldots,j_k})=\b^i_{j_1,\ldots,j_k}.
\end{equation}
  
\begin{lemma}
The map $\Db:\Fc(\Nbo)\ra   \Fc(\Nbo)\ot \Fc(\Gb^\dagger_0)$ defined by
\begin{align} \label{coaction-F-1-F-2}
\Db(f)(\nu,\lambda)=f(\nu\lt \lambda)  ,
\end{align}
is a coaction and makes $\Fc(\Nbo)$ a $\Fc(\Gb^\dagger_0)$ a comodule Hopf algebra.
\end{lemma}
\begin{proof}
Denoting the inclusion $\Fc(\Nbo)\hookrightarrow \Fc(\Gb^\dagger)$ by
\begin{equation}
\hat{f}(\psi)= f(\nu_\psi), 
\end{equation}
one observes that
\begin{align}
&\Db(f)(\nu,\lambda)=f(\nu\lt \lambda)=\hat f(\nu\circ\lambda)=\\ &\hat{f}\ps{1}(\nu)\hat f\ps{2}(\lambda)=\pi_2(\hat f\ps{1})\ot \pi_1(\hat f\ps{2})(\nu,\lambda).
\end{align}
Since $\lt$ is a group action, it is easily seen that $\Db$ is a coaction.
The fact that $\Db$ preserves the product,
$$
\Db(f^1f^2)=\Db(f^1)\Db(f^2),
$$
is also clear.
 
 Next we show that
$\Fc(\Nbo)$ is $\Fc(\Gb^\dagger_0)$-comodule coalgebra.  Indeed,
\begin{align}
\begin{split}
&f\ns{0}\ps{1}\ot f\ns{0}\ps{2}\ot f\ns{1}(\nu_1,\nu_2,\lambda)=\\
&f\ns{0}\ot f\ns{1}(\nu_1\circ\nu_2,\lambda)= f((\nu_1\circ\nu_2)\lt \lambda)=\\
&f((\nu_1\lt(\nu_2\rt\lambda))\circ (\nu_2\lt \lambda))=\\
&f((\nu_1\lt\lambda)\circ (\nu_2\lt \lambda))=\\
&f\ps{1}(\nu_1\lt\lambda)f\ps{2}(\nu_2\lt \lambda)= \\
&f\ps{1}\ns{0}(\nu_1)f\ps{1}\ns{1}(\lambda)f\ps{2}\ns{0}(\nu_2)f\ps{2}\ns{1}(\lambda)=\\
&f\ps{1}\ns{0}\ot f\ps{2}\ns{0}\ot f\ps{1}\ns{1}f\ps{2}\ns{1}(\nu_1,\nu_2,\lambda).
\end{split}
\end{align}
The last required condition is also satisfied:
\begin{align}
\ve(f\ns{0})f\ns{1}(\lambda)=\Db(f)(e,\lambda)=f(e\lt \lambda)=f(e)=\ve(f)1_{\Fc(\Gb_1G^+)}(\lambda).
\end{align}
\end{proof}

We note that, with the notation introduced above, the following identity holds:
\begin{equation}
\check\a^i_s\hat{\a}^s_{j_1,\ldots, j_k}=\b^i_{j_1,\ldots,j_k}\;.
\end{equation}
\smallskip

Since the action of $\Fc(\Nbo)$ on $\Fc(\Gb^\dagger_0)$ is trivial, that is  given by $\ve$ , we see that all conditions of   matched pair of Hopf algebras  are satisfied and we have the Hopf algebra $\Fc(\Gb^\dagger_0)\acl\Fc(\Nbo)$. Moreover, as an algebra $\Fc(\Gb^\dagger_0)\acl\Fc(\Nbo)$ is just $\Fc(\Gb^\dagger_0)\ot\Fc(\Nbo)$ and as a coalgebra it is $\Fc(\Gb^\dagger_0)\cl\Fc(\Nbo)$.
We will therefore adopt the latter notation.

There also is a natural left coaction of $\Fc(\Nbo)$ on $\Fc(\Gb^\dagger_0)$, which
we record below.

\begin{lemma}
The map $\Db_L:\Fc(\Nbo)\ra \Fc(\Gb^\dagger_0)\ot \Fc(\Nbo)$ defined by
\begin{align}
 \Db_L(f)(\lambda,\nu)=\Ic_2(f)(\lambda,\nu).
\end{align} 
 defines a left coaction, which satisfies
\begin{equation}\label{left-coaction-gamma}
\Db_L(\a^i_{j_1, \ldots,j_k})=\a^i_s\ot\a^s_{j_1,\ldots,j_k}.
\end{equation}
\end{lemma}
\begin{proof}
It suffices to prove that
\eqref{left-coaction-gamma} holds, and this is straightforward:
\begin{equation*}
\Db_L(\a^i_{j_1,\ldots,j_k})(\lambda,\nu)=\b^i_{j_1,\ldots,j_k}(\lambda\circ\nu)=\b^i_s(\lambda)\b^s_{j_1,\ldots,j_k}(\nu)=\a^i_s(\lambda)\a^s_{j_1,\ldots,j_k}(\nu).
\end{equation*}
\end{proof}

We now define a natural map $\Phi:\Fc(\Gb^\dagger)\ra \Fc(\Gb^\dagger_0)\ot \Fc(\Nbo)$ by
\begin{align}
 \Phi(f)(\lambda,\nu)=f(\lambda\circ\nu)=\pi_1(f\ps{1})\ot \pi_2(f\ps{2})(\lambda,\nu).
\end{align}

\begin{lemma}
The map $\Phi:\Fc(\Gb^\dagger)\ra \Fc(\Gb^\dagger_0)\ot \Fc(\Nbo)$ is an isomorphism of  algebras.
\end{lemma}
\begin{proof}
 $\Phi$ is obviously linear, and 
 \begin{equation}
\Phi^{-1}(f^1\ot f^2)= \Ic_1(f^1)S(\Ic_1(f^2\ns{-1}))\Ic_2(f^2\ns{0}).
\end{equation}
defines a two sided  inverse for $\Phi$.
Indeed,
\begin{align}\label{PHI}
\begin{split}
&\Phi(\b^i_j)= \a^i_j\ot 1, \quad \Phi(\b^{-1}) =\a^{-1}\ot 1,\quad \Phi(\b^i_{j_1,\ldots,j_k})= \a^i_s\ot\a^s_{j_1,\ldots,j_k},\\
&\Phi^{-1}(\a^i_j\ot 1)=\b^i_j,\quad \Phi^{-1}(\a^{-1}\ot 1)= \b^{-1},\quad \Ph^{-1}(1\ot \b^i_{j_1,\ldots j_k})= S(\b^i_s)\b^s_{j_1,\ldots,j_k}.
\end{split}
\end{align}
Since both maps $\Phi$ and $\Phi^{-1}$ are algebra maps, the claim follows.
\end{proof}

\begin{proposition}\label{proposition-decompos-F-1-F-2}
The map $\Phi:\Fc(\Gb^\dagger)\ra \Fc(\Gb^\dagger_0)\ot \Fc(\Nbo)$ is an isomorphism of 
Hopf algebras. 
\end{proposition}
\begin{proof}
Both sides being commutative, it suffice to check the compatibility of $\Phi$ with the coalgebra structures. 

Let $f\in \Fc(\Gb^\dagger)$ be of the form
$f=f^1f^2$, with $f_1\in \Fc(\Gb^\dagger_0)$, $f_2\in \Fc(\Nbo)$, 
which means that $\Phi(f)=f^1\cl f^2$.
The  comultiplication of $\Fc(\Gb^\dagger_0)\cl\Fc(\Nbo)$ is given by 
\begin{align}
\begin{split}
&\D_{\Fc_1(\Gb^\dagger)\cl\Fc_2(\Gb^\dagger)}(\Phi(f))(\lambda_1,\nu_1;\lambda_2,\nu_2)\\
&=\big(f^1\ps{1}\cl f^2\ps{1}\ns{0}\ot f^1\ps{2}f^2\ps{1}\ns{0}\cl f^2\ps{2}\big)(\lambda_1,\nu_1;\lambda_2,\nu_2)\\
&=f^1\ps{1}(\lambda_1)f^2\ps{1}\ns{0}(\nu_1)f^1\ps{2}(\lambda_2)f^2\ps{1}\ns{0}(\lambda_2)f^2\ps{2}(\nu_2)=\\
&=(f^1\ps{1}(\lambda_1)f^1\ps{2}(\lambda_2))\; (f^2\ps{1}\ns{0}(\lambda_2)f^2\ps{1}\ns{0}(\nu_1)) \;f^2\ps{2}(\nu_2)\\
&=f^1(\lambda_1\circ\lambda_2) f^2\ps{1}(\nu_1\lt \lambda_2)f^2\ps{2}(\nu_2)\\
&=f^1(\lambda_1\circ\lambda_2) f^2((\nu_1\lt \lambda_2)\circ\nu_2).
\end{split}
\end{align}
On the other hand one uses \eqref{decomposition-G-1-G-2} and the fact that $\nu\rt\lambda=\lambda$ for any $\lambda\in \Gb^\dagger_0$ and any $\nu\in G_2^+$ to see that
\begin{align}
\begin{split}
&\Phi\ot\Phi (\D(f))(\lambda_1,\nu_1;\lambda_2,\nu_2)\\
 &=\Phi(f\ps{1})(\lambda_1,\lambda_2) \Phi(f\ps{2})(\lambda_2, \nu_2))\\
 &=f\ps{1}(\lambda_1\circ\nu_1)f\ps{2}(\lambda_2\circ\nu_2)\\
 &=f(\lambda_1\circ\nu_1\circ\lambda_2\circ\nu_2)\\
 &=f(\lambda_1\circ\lambda_2\circ (\nu_2\lt \lambda_1)\circ \nu_2)\\
&=f^1(\lambda_1\circ\lambda_2)f^2(\nu_2\lt \lambda_1)\circ \nu_2).
\end{split}
\end{align}
Finally it is easy to see that
\begin{equation}
\ve_{\Fc_1(\Gb^\dagger)\cl\Fc_2(\Gb^\dagger)}(\Phi(f))= f^1(e)f^2(e)=f(e)=\ve_{\Fc(\Gb^\dagger)}(f).
\end{equation}
\end{proof}
Via the above isomorphism  we identify
\begin{equation}
\b^i_j =\a^i_j\cl 1,\quad \b^{-1} =\a^{-1}\cl 1, \quad \b^i_{j_1,\ldots j_k}=\a^i_{s}\cl \a^s_{j_1,\ldots,j_k}
\end{equation}

For $\a^i_{j_1, \ldots,j_k}$ we have
\begin{equation}
\Db(\a^i_{j_1, \ldots,j_k})= \a^r_{s_1,\ldots,s_k}\ot S(\a^i_r)\a^{s_1}_{j_1}\cdots \a^{s_k}_{j_k}.
\end{equation}

Let $\Fg\Fl_n$ be the Lie algebra $\GL_n(\Rb)$. One defines 
a Hopf pairing between $\Fc(\Gb^\dagger_0)$ and $\Uc(\Fg\Fl_n)$
by extending the natural pairing
 \begin{equation}
 \langle f, Y\rangle=Y(f)=\dt f( \exp(tY)), \qquad Y\in \Fg\Fl_n,\lambda\in \Gb^\dagger_0 .
 \end{equation}
This means that
 \begin{align}\label{pairing-properties}
\begin{split}
 &\langle f^1f^2, u\rangle=\langle f^1, u\ps{1}\rangle \langle f^2, u\ps{2}\rangle, \quad \langle f, u^1u^2\rangle=\langle f\ps{1}, u^1\rangle \langle f\ps{2}, u^2\rangle,\\
&\langle f, 1\rangle=\ve(f),\qquad \langle 1, u\rangle=\ve(u),\qquad\langle f, S(u)\rangle=\langle S(f), u\rangle.
 \end{split}
\end{align}

 We now define the action $\Fg\Fl_n\ot \Fc(\Nbo)\ra \Fc(\Nbo)$ by
\begin{align}\label{action-Y-F2}
Y\rt f= f\ns{0}Y(f\ns{1}) .
\end{align}
We denote the standard basis of $\Fg\Fl_n$
 by $Y^i_j$, $1\le i,j\le n$.  One  first observes that $Y^i_j(\a^p_q)= \d^i_q\d^p_j$. 
 Then we use  the Hopf pairing properties \eqref{pairing-properties} to see that
 \begin{align}\label{action-gl-F}
 \begin{split}
 &Y^i_j\rt \a^p_{q_1,\ldots, q_m}=\sum_s \d^i_{q_s}\a^p_{q_1,\ldots, j,\ldots q_m}-\d^p_j \a^i_{q_1, \ldots, q_m},
 \end{split}
 \end{align}

By restricting the action of $\Gb^\dagger$ on $\Tbo$  to $\Gb^\dagger_0$ we get the coaction
\begin{align}\label{coaction-F-1-V-new}
\begin{split}
&{V}\ra {V}\ot \Fc(\Gb^\dagger_0),\\
&X_k\ra X_s\ot \a^s_k.
\end{split}
\end{align}
We use this coaction to define an action of $\Fg\Fl_n$ on ${V}$ via
 \begin{equation}\label{action-gl-V}
  Y\rt X=\langle X\ns{0}\,,\, Y\rangle X\ns{1}.
 \end{equation}
 Note that the action of $Y^i_j$ on $X_k$ is indeed the natural action of $\Fg\Fl_n$ on $\Rb^n$, \ie
 \begin{equation}\label{action-Y-V}
 Y^i_j\rt X_k\; = \;\d^i_k\; X_j.
 \end{equation}


\section{Hopf cyclic cohomology of $\Kc_n$}\label{S3}
For the reader's convenience we recall two basic notions.

Given a Hopf algebra $\Hc$, a character $\d: \Hc \ra \Cb$ and 
a group-like element $\s\in \Hc$, the pair $(\d,\s)$ is called a
{\em modular pair in involution} if
 \begin{equation}\label{def-MPI}
 \d(\s)=1, \quad \text{and}\quad  S_\d^2=Ad_\s,
 \end{equation}
where  $Ad_\s(h)= \s h\s^{-1}$ and  $S_\d(h)=\d(h\ps{1})S(h\ps{2})$, $h \in \Hc$.  
To such a datum was associated in~ \cite{cm4} a cyclic module whose cohomology,
called Hopf cyclic, is denoted $HC^\bullet (\Hc ;\, ^\s\Cb_\d)$.

Let now $\Fg$ be a finite-dimensional Lie algebra and let $\Fc$ be a $\Fg$-Hopf algebra
(\cf~\cite{RS}), on which $\Fg$ coacts via $\Db_\Fg:\Fg\ra \Fg\ot \Fc$. 
The modular character of $\d:\Fg\ra \Cb$, 
\begin{equation*}
 \d(X)= \Trace (\ad_X)  , \qquad X \in \Fg ,
\end{equation*}
 extends to a character of $\Uc(\Fg)$. One then further extends $\d$ to a
 character of $\Fc\acl \Uc(\Fg)$ by
\begin{equation*}
\d(f\acl u)=\ve(f)\d(u).
\end{equation*}
 In a dual fashion, one defines a group-like  element in $\Fc$ as follows.
The (first-order) matrix coefficients $ f^i_j\in \Fc$ of the coaction $\Db_\Fg$ are
given by the equation
\begin{equation*}
\Db_\Fg(X_j)= \sum_{i=1}^n X_i\ot f_j^i , \qquad n=\dim\Fg ;
\end{equation*}
they satisfy the relation
 \begin{equation*} 
\D(f_i^j)=\sum_{k=1}^n f_k^j\ot f_i^k .
\end{equation*}

\begin{equation}
\s_F:= \det (f^i_j)=\sum_{\pi\in S_n}(-1)^\pi f^1_{\pi(1)}\cdots f^n_{\pi(n)}.
\end{equation}
 One then defines a group-like element in $F\acl \Uc(\Fg)$ by setting
\begin{equation*}
\s:= \s_F\acl 1 , \quad \text{where} \quad 
\s_F:=\sum_{\pi\in S_n}(-1)^\pi f^1_{\pi(1)}\cdots f^n_{\pi(n)}.
\end{equation*}
It is shown in \cite[Theorem 3.2]{RS} that the $(\d, \s)$ defines a modular pair in involution for the Hopf algebra $F\acl \Uc(\Fg)$.

\bigskip
We now return to the $ {V}$-Hopf algebra ${\Fc_\Kc}$ of \S1, equipped with
the action and  the coaction defined in \eqref{action-V-F(G)} and \eqref{coaction-V}.
Since the Lie algebra $ {V}$ is commutative, $\d$ coincides  with $\ve$ and hence $(\ve, \s)$
is a modular pair in involution for ${\Fc_\Kc}\acl \Uc(V)$. Denoting by
 $^\s\Cb$ the one-dimensional left comodule and right module  over ${\Fc_\Kc}\acl \Uc$
 determined by the group-like $\s$ and the character $\ve$ respectively, we are thus in a position
 to form the Hopf cyclic cohomology $HC^{\bullet}({\Fc_\Kc}\acl \Uc, ^\s\Cb)$ of the
 canonically associated $(b, B)$-bicomplex (\cf~\cite{cm2, cm3, cm4}).

In order to compute this cohomology, we employ a quasi-isomorphic 
bicomplex, which takes advantage of the bicrossed product structure of the Hopf
algebra $\Kc_n$ as well as of the particular nature of its components. 
Referring the reader to \cite{mr09, mr11,kr2} for details concerning 
the intermediate steps, we proceed to describe the resulting quasi-isomorphic bicomplex.

The Lie  algebra $V$ admits the following right
 action on ${\Fc_\Kc}^{\ot q}$ 
 \begin{align} \label{bullet}
&X\bullet (f^1 \ot \cdots \ot f^q) =  \\ \notag
&X\ps{1}\ns{0}\rt f^1\ot X\ps{1}\ns{1}(X\ps{2}\ns{0}\rt f^2)\odots X\ps{1}\ns{q-1}\dots X\ps{q-1}\ns{1}(X\ps{q}\rt f^q),
\end{align}
On the other hand, since ${\Fc_\Kc}$ is commutative, the coaction $\Db: {V}\ra  {V}\ot {\Fc_\Kc}$,
 extends  from $ {V}$ to
 a unique coaction $\Db_{ {V}}:\wg^p {V}\ra \wg^p {V}\ot {\Fc_\Kc}$. After tensoring it with the right coaction of $^\s{\Cb}$ we obtain the coaction
\begin{align}
\begin{split} \label{wgcoact}
&\Db_{^\s{\Cb}\ot\wdg {V}}(\one\ot X^1\wdots X^q)\,=\\
&\,\one\ot X^1\ns{0}\wdots X^q\ns{0}\ot
\s^{-1}X^1\ns{1}\dots X^q\ns{1} .
\end{split}
\end{align}

Let $\{X_i \}_{1\leq i\leq n}$ be  the basis for $ {V}$ and let $\{\t^j\}_{1\leq j \leq n}$ denote
the dual basis of $ {V}^\ast$.   One defines the dual
 left coaction on $\Db_ {V}^\ast: {V}^\ast\ra  {V}^\ast\ot {\Fc_\Kc}$ by
\begin{equation}\label{coaction-g*-F}
\Db_ {V}^\ast(\t^i)= \sum_j \b^i_j\ot \t^j , \quad \text{where} \quad \Db_ {V}(X_i)=X_j\ot \b^j_i .
\end{equation} 
We extend this coaction on $\wdg ^\bullet  {V}^\ast$ diagonally and observe that the  result is a left coaction just because ${\Fc_\Kc}$ is commutative. For later use we  record below that if
$\om:= \t^{i_1}\wdots \t^{i_k}$ then
\begin{align}\label{coaction-g-ast}
\begin{split}
\om\ns{-1}\ot \om\ns{0}= \sum_{1\le l_j\le m} f_{l_1}^{i_1}\cdots f_{l_k}^{i_k}\ot \t^{l_1}\wdots \t^{l_k}.
\end{split}
\end{align}
One uses  the antipode of ${\Fc_\Kc}$ to turn it into a right coaction
 $\Db_{\wdg {V}^\ast}: \wdg^p {V}^\ast\ra \wdg^p {V}^\ast\ot {\Fc_\Kc}$, as follows:
 \begin{align}\label{coaction-F-1-W}
 \Db_{\wdg {V}^\ast}( \om)=  \om\ns{1}\ot S(\om\ns{-1}).
\end{align}

We now use the above ingredients to build the following bicomplex:
 \begin{align}\label{bicomp-V*}
 \begin{split}
 \xymatrix{  \vdots & \vdots
 &\vdots &&\\
  \wdg^2 {V}^\ast  \ar[u]^{\p_{ {V}^\ast}}\ar[r]^{b^\ast_{\Fc_\Kc}~~~~~~~}&   \wdg^2 {V}^\ast\ot{\Fc_\Kc} \ar[u]^{\p_{ {V}^\ast}} \ar[r]^{b^\ast_{\Fc_\Kc}}&  \wdg^2 {V}^\ast\ot{\Fc_\Kc}^{\ot 2} \ar[u]^{\p_{ {V}^\ast}} \ar[r]^{~~~~~~~~~b^\ast_{\Fc_\Kc}} & \hdots&  \\
   {V}^\ast  \ar[u]^{\p_{ {V}^\ast}}\ar[r]^{b^\ast_{\Fc_\Kc}~~~~~}&   {V}^\ast\ot{\Fc_\Kc} \ar[u]^{\p_{ {V}^\ast}} \ar[r]^{b^\ast_{\Fc_\Kc}}&    {V}^\ast\ot {\Fc_\Kc}^{\ot 2} \ar[u]^{\p_{ {V}^\ast}} \ar[r]^{~~~~~b^\ast_{\Fc_\Kc} }& \hdots&  \\
   \Cb\ar[u]^{\p_{ {V}^\ast}}\ar[r]^{b^\ast_{\Fc_\Kc}~~~~~~~}&  {\Fc_\Kc} \ar[u]^{\p_{ {V}^\ast}}\ar[r]^{b^\ast_{\Fc_\Kc}}&  {\Fc_\Kc}^{\ot 2} \ar[u]^{\p_{ {V}^\ast}} \ar[r]^{~~~~~b^\ast_{\Fc_\Kc}} & \hdots&. }
\end{split}
\end{align}
 The vertical coboundary $\p_{ {V}^\ast}:C^{p,q}\ra C^{p,q+1}$ is the Lie algebra
cohomology coboundary of the Lie algebra $ {V}$ with coefficients in ${\Fc_\Kc}^{\ot p}$,
 where the action of $ {V}$ is given by
\begin{multline}
( \one\ot f^1\odots f^p) \blacktriangleleft X=-\; \one\ot X\bullet(f^1\odots f^p).
\end{multline}
The horizontal $b$-coboundary $b^*_{\Fc_\Kc}$ has the expression
\begin{align}
\begin{split}
&b^\ast_{\Fc_\Kc}(\a\ot f^1\odots f^q)=\\
&\a\ot 1\ot f^1\odots f^q +\sum_{i=1}^q(-1)^i \a\ot f^1\odots \D(f^i)\odots f^q+\\
&(-1)^{q+1}\ot \a\ns{1}\ot f^1\odots f^q\ot \a\ns{-1}.
\end{split}
\end{align}
\medskip

At this stage we recall that by Proposition \eqref{proposition-decompos-F-1-F-2}
the Hopf subalgebra has a further factorization,
\begin{equation}
\Fc_\Kc\cong \Fc(\Gb^\dagger_0)\cl \Fc(\Nbo).
\end{equation}
This allows to apply the same treatment alluded to above to each row of
the bicomplex  \eqref{bicomp-V*}. 

Let us describe the bicomplex which
computes the cohomology of the $p$th row. 
To simplify the notation, in what follows 
we  abbreviate $\Fc_1:=\Fc(\Gb^\dagger_0)$ and $\Fc_2:=\Fc(\Nbo)$.

Diagrammatically,
\begin{align}\label{bicomp-F-1-F-2}
\begin{xy} \xymatrix{  \vdots & \vdots
 &\vdots &&\\
  \wdg^p {V}^\ast\ot \Fc_2^{\ot 2}  \ar[u]^{b^\ast_{\Fc_1}}\ar[r]^{b^\ast_{\Fc_1}~~~~~~~}&   \wdg^p {V}^\ast \ot \Fc_2^{\ot 2}\ot \Fc_1\ar[u]^{b^\ast_{\Fc_2}} \ar[r]^{b^\ast_{\Fc_1}}&  \wdg^p {V}^\ast \ot \Fc_2^{\ot 2}\ot \Fc_1^{\ot 2}\ar[u]^{b^\ast_{\Fc_2}} \ar[r]^{~~~~~~~~~b^\ast_{\Fc_1}} & \hdots&  \\
     \wdg^pV^\ast\ot \Fc_2 \ar[u]^{b^\ast_{\Fc_2}}\ar[r]^{b^\ast_{\Fc_1}~~~~~}&   \wdg^pV^\ast\ot \Fc_2 \ot \Fc_1\ar[u]^{b^\ast_{\Fc_2}} \ar[r]^{b^\ast_{\Fc_1}}&     \wdg^pV^\ast \ot \Fc_2 \ot \Fc_1^{\ot 2} \ar[u]^{b^\ast_{\Fc_2}} \ar[r]^{~~~~~b^\ast_{\Fc_1} }& \hdots&  \\
   \wdg^pV^\ast\ar[u]^{b^\ast_{\Fc_2}}\ar[r]^{b^\ast_{\Fc_1}~~~~~~~}& \wdg^pV^\ast \ot \Fc_1 \ar[u]^{b^\ast_{\Fc_2}}\ar[r]^{b^\ast_{\Fc_1}}&  \wdg^pV^\ast \ot \Fc_1^{\ot 2} \ar[u]^{b^\ast_{\Fc_2}} \ar[r]^{~~~~~b^\ast_{\Fc_1}} & \hdots&. }
\end{xy}
\end{align}
The $q$th row above is the Hochschild complex of
the coalgebra $\Fc_1$ with coefficients in the comodule $\wdg^pV^\ast\ot \Fc_2^{\ot q}$ defined by
\begin{align}
\begin{split}
\Db_{{V}^\ast\ot\Fc_2}:\wdg^qV^\ast\ot \Fc_2^{\ot \bullet}\ra \Fc_1\ot \wdg^qV^\ast\ot \Fc_2^{\ot \bullet},\\
\Db(\om\ot \td f)= \om\ns{-1}S(\td f\ns{1})\ot \om\ns{0}\ot \td f\ns{0} ,
\end{split}
\end{align}
 where we  use the natural left  coaction of $\Fc_1$ on $\wdg^pV^\ast$ defined in \eqref{coaction-g-ast}.
 
In the above bicomplex we also use the coaction of $\Fc_1$ on $\Fc_2$ defined in \eqref{coaction-F-1-F-2} and extend  it on $\td f=f^1\odots f^q\in \Fc_2^{\ot q}$ by
\begin{equation}
\td f\ns{0}\ot \td f\ns{1}= f^1\ns{0}\odots f^q\ns{0}\ot f^1\ns{1}\cdots f^q\ns{1},
\end{equation}
The columns of the bicomplex are just the Hochschild complexes of the coalgebra $\Fc_2$ with trivial coefficients $\wdg^p {V}^\ast\ot \Fc_1^{\ot \bullet}$.

\begin{proposition}\label{Hoch}
The cohomology of the $q$th row of the bicomplex  \eqref{bicomp-F-1-F-2} is concentrated 
in the first column and coincides  with $(\wdg^p {V}^\ast\ot \Fc_2^{\ot q})^{\Fg\Fl_n}$.
\end{proposition}
\begin{proof}
The Lie algebra $\Fg\Fl_n$, viewed as a subalgebra of formal vector fields on $\Rb^n$,
 acts naturally on both $V^\ast$ and on $\Fc_2$; it is this standard action which appears
in the above statement.  
 
Set \, $Z^{p,q}= \wdg^p {V}^\ast\ot \Fc_2^{\ot q}$.
The $q$th row is the Hochschild complex of the $\Fc_1$ with  coefficients in $Z^{p,q}$. We use the identification of $\Fc_1$ with $\Pc(\GL_n)$ and, since $\GL_n$ is reductive, we are in a position  
to apply \cite[Theorem 4.8]{RS} to
infer that the cohomology of the row is concentrated in the $0$th  cohomology group, in other words that
\begin{equation}
H^k(\Fc_1,Z^{p,q})\;=\; 0, \quad k>0, \quad H^0(\Fc_1\;,\;Z^{p,q})\;=\;({Z^{p,q}})^{\Fc_1} .
\end{equation}
Here $(Z^{p,q})^{\Fc_1}$ is the space of coinvariants elements with respect to the coaction  
$\Fc_1$ on $Z^{p,q}$ defined in \eqref{coaction-F-1-W}, i.e.,
\begin{equation}
(Z_{p,q})^{\Fc_1}=\{\om\ot \td f\;\mid\; \Db_{{V}^\ast\ot \Fc_2}(\om\ot \td f)=1\ot \om\ot \td f\}
\end{equation}

The action of  $\Fg\Fl_n$  on $\Fc_2$ is that defined in \eqref{action-gl-F}. 
The action of $\Fg\Fl_n$ on ${V}^\ast$ comes from the coaction of $\Fc_1$ on ${V}^\ast$ 
defined by \eqref{coaction-g*-F}. Explicitly,,
\begin{equation}
\t^k\lt Y^i_j\;=\; \langle \b^k_\ell\,,\, Y^i_j\rangle\; \t^\ell\; =\; \d^k_j\d^i_\ell\;\t^\ell\;=\; \d^k_j\;\t^i,
\end{equation}
which is transpose of the standard action of $\Fg\Fl_n$ on ${V}$ defined in \eqref{action-gl-V}.

Let us show that this space of coinvariants coincides with the invariants under 
$\Fg\Fl_n$. 
The right action of $\Fg\Fl_n$ on $Z^{p,q}$ is given by 
\begin{align}
 (\om\ot \td f)\lt Y= \langle (\om\ot \td f)\ns{-1}\,,\, Y\rangle \,  (\om\ot \td f)\ns{0},
\end{align}
and we need 
to check that  $\om\ot \td f\in( {Z^{p,q}})^{\Fc_1}$ if and only if $\om\ot \td f\in (Z^{p,q})^{\Fg\Fl_n}$,
or equivalently,
\begin{equation}
\om\ot \td f\in (Z^{p,q})^{\Fc_1}\, \iff \,\om\lt Y\ot \td f- \om\ot Y\rt \td f=0.
\end{equation}
\end{proof}

\begin{theorem}\label{theorem-chern}
The periodic Hopf cyclic cohomology $HP^{\bullet}(\Kc_n ;\, ^{\s^{-1}}\Cb)$ is isomorphic to the
truncated ring of Chern classes $P_{2n}[c_1, \ldots , c_n]$.
 \end{theorem}

\begin{proof}
By Proposition \ref{Hoch} the total cohomology of the bicomplex \eqref{bicomp-F-1-F-2} 
reduces to  the cohomology of the following complex
\begin{equation}
\xymatrix{(\wdg^p {V}^\ast)^{\Fg\Fl_n}\ar[rr]^{b_{\Fc_2}^\ast} & &(\wdg^p {V}^\ast\ot \Fc_2)^{\Fg\Fl_n}\ar[rr]^{b_{\Fc_2}^\ast}&& \cdots }
\end{equation}
We still need to calculate the total cohomology of \eqref{bicomp-V*}. 
Via the above identification, that bicomplex is quaisi-isomorphic with the following one,
 \begin{align}\label{bicomp-V*-rel}
  \begin{split}
 \xymatrix{  \vdots & \vdots
 &\vdots &&\\
  (\wdg^2 {V}^\ast)^{\Fg\Fl_n}  \ar[u]^{\p_{ {V}^\ast}}\ar[r]^{b^\ast_{\Fc_2}~~~~~~~}& (  \wdg^2 {V}^\ast\ot{\Fc_2})^{\Fg\Fl_n} \ar[u]^{\p_{ {V}^\ast}} \ar[r]^{b^\ast_{\Fc_2}}&  (\wdg^2 {V}^\ast\ot{\Fc_2}^{\ot 2} )^{\Fg\Fl_n}\ar[u]^{\p_{ {V}^\ast}} \ar[r]^{~~~~~~~~~b^\ast_{\Fc_2}} & \hdots&  \\
   ({V}^\ast )^{\Fg\Fl_n} \ar[u]^{\p_{ {V}^\ast}}\ar[r]^{b^\ast_{\Fc_2}~~~~~}&   ({V}^\ast\ot{\Fc_2})^{\Fg\Fl_n} \ar[u]^{\p_{ {V}^\ast}} \ar[r]^{b^\ast_{\Fc_2}}&  (  {V}^\ast\ot {\Fc_2}^{\ot 2})^{\Fg\Fl_n} \ar[u]^{\p_{ {V}^\ast}} \ar[r]^{~~~~~b^\ast_{\Fc_2} }& \hdots&  \\
   \Cb\ar[u]^{\p_{ {V}^\ast}}\ar[r]^{b^\ast_{\Fc_2}~~~~~~~}& ( {\Fc_2})^{\Fg\Fl_n} \ar[u]^{\p_{ {V}^\ast}}\ar[r]^{b^\ast_{\Fc_2}}&(  {\Fc_2}^{\ot 2} )^{\Fg\Fl_n}\ar[u]^{\p_{ {V}^\ast}} \ar[r]^{~~~~~b^\ast_{\Fc_2}} & \hdots&. }
\end{split}
\end{align}
One uses \eqref{decopmpose-G+},  \eqref{decopmpose-G2+},  and \eqref{F-2-alpha} on one hand and \cite[Proposition 2.1, Definition 2.4]{mr09} on the other hand to observe that
\begin{equation}
\Fc_2=\Fc_\Hc \, ,
\end{equation}
 where $\Fc_\Hc \subset \Hc_n$ is the Hopf subalgebra, 
 denoted by $\Fc (\Nbo)$ in \cite{mr09}, such that
 $\Hc_n^{\rm cop}=\Fc_\Hc\acl \Uc(\Fg\Fl_n^{\rm aff})$.

After this identification one applies \eqref{action-Y-F2}, \eqref{coaction-F-1-V-new}, and \eqref{action-Y-V} to observe that the bicomplex \eqref{bicomp-V*-rel} is identified with the bicomplex (4.12) in \cite{mr11} for $\Fh=\Fg\Fl_n$, or alternately with the bicomplex (3.42) in \cite{mr09}. The total cohomology of the latter bicomplex  is computed 
in \cite[Theorem 3.25]{mr09}, and  shown to be isomorphic to $P_{2n}[c_1, \ldots , c_n]$.
\end{proof}

There is an alternative way to formulate the above result, which relies on
identifying, as coalgebras, 
 $\Kc_n$ and the quotient coalgebra $\Qc_n:=\Hc_n\ot_{\Uc(\Fg\Fl_n)}\Cb$.    

First, one identifies the copposite coalgebra $\Qc_n^{\cop}$  to
 $\Fc_\Hc\acl \Uc(\Fg\Fl_n^{\rm aff})\ot_{\Uc(\Fg\Fl_n)}\Cb$ as the $\Uc(\Fg\Fl_n)$, 
is isomorphic to the crossed  product coalgebras $\Fc_\Hc\cl \Uc(V)$, via the map
 
\begin{align}\label{map-betta}
\begin{split}
\kappa_\Hc: \Hc_n^{\cop}\ot_{\Uc(\Fg\Fl_n)}\Cb &= (\Fc_\Hc\acl \Uc(\Fg\Fl_n^{\rm aff}))\ot_{\Uc(\Fg\Fl_n)}\Cb 
\ra \Fc_\Hc\cl \Uc(V) \\
&  \kappa_\Hc(f\acl XY\ot_{\Uc(\Fg\Fl_n)}1) =\ve(Y)f\cl X ;
 \end{split}
\end{align}
here $\Fc_\Hc$ coacts on  $\Uc(V)$ via  its coaction on $\Uc(\Fg\Fl_n^{\rm aff})$  followed by the 
projection $\pi:\Uc(\Fg\Fl_n^{\rm aff})\ra \Uc(V)$ that is defined by $\pi(XY)=\ve(Y)X$, and is a map of coalgebras.  It is clear that $\kappa_\Hc$ is an isomorphism. 

We next consider the map $\kappa^\dagger: \Fc(\Gb^\dagger)\cl \Uc(V)\ra \Fc_\Hc\cl \Uc(V)$ by
the formula
\begin{align} \label{kappa}
 \kappa^\dagger (\b^i_{j_1,j_2,\ldots j_s}\cl u)= \d^i_s\a^s_{j_1,j_2,\ldots,j_s}  \cl u ,
\end{align}
where $\d^i_j$ is the  Kronecker's delta tensor  and $\a^i_j:=\d^i_j$.   
This map is quite natural, being the same as $r_\Hc \ot\Id$, where
 $r_\Hc:\Fc(\Gb^\dagger)\ra \Fc(\Nbo)\cong\Fc_\Hc$ 
is the restriction map, dual to the inclusion $\Nbo\hookrightarrow \Gb^\dagger$.

\begin{theorem} \label{thm-alt-chern}
The map $\kappa :=\, \kappa_\Hc^{-1}\circ  \kappa^\dagger:\Fc(\Gb^\dagger)\cl \Uc(V)\ra 
 \Hc_n^{\cop}\ot_{\Uc(\Fg\Fl_n)}\Cb $ is a morphism of coalgebras which
 induces a quasi-isomorphism of Hochschild cohomology complexes
\begin{align*}
 \{\;^{\s}\Cb\ot {\Kc^\cop}^{\ot \bullet} ,\,  b\} \ra 
 \{\Cb_\d \ot_{\Uc(\Fg\Fl_n)} {\Qc_n^{\cop}}^{\ot \bullet} , \, b\} .
\end{align*}
This in turn yields an isomorphism $HC^\ast(\Kc; \;^{\s^{-1}}\Cb)\cong HC(\Hc_n, GL_n;\Cb_\d)$. 
\end{theorem}

\begin{proof}
Proposition \ref{proposition-decompos-F-1-F-2} guarantees that  
$\kappa^\dagger$ is a map of coalgebras. Using the definition \eqref{kappa},
which in particular implies that
$\kappa^\dagger(\b^i_j)=\d^i_j$, it is easy to check that $\kappa$ induces
a chain map at the level of Hochchild complexes. 
 
 The second claim follows by combining the following two facts: 
 the vanishing of the Connes boundary map $B$ at the level of both Hochschild complexes;
 the vanishing of the Hochschild cohomology groups outside degree $n$ 
of both sides, ensured by Theorem \ref{theorem-chern}, resp. \cite[Theorem 3.25]{mr09}. 
\end{proof}

\section{Geometric representation of the Hopf cyclic Chern classes} \label{S4}

In order to exhibit concrete cocycles representing a basis of $HP^*(\Kc_n ;\, ^\s\Cb)$,
we take the same approach as in~\cite{M-GC, M-EC}. 
The gist of that construction is summarized below. 
\smallskip
 
With $M=\Rb^n$ and $\Gb = \Diff(\Rb^n)^\d$, let 
$\{\Om^\bullet (|\bar\triangle_{\Gb} M|), d \}$ be
the complex of Dupont's~\cite{Dupont} complex of de Rham simplicial
compatible forms. We will actually work with its homogeneous version
$\{\Om^\bullet (|\bar\triangle_{\Gb} M|), d \}$. The identification
between compatible forms $\om = \{\om_p\}_{p\geq 0}$ in the first
complex and their homogeneous counterpart $\bar{\om} = \{\om_p\}_{p\geq 0}$
in the second, \ie satisfying
\begin{align*}  
 \bar{\om} (\tb; \rho_0 \rho, \ldots , \rho_p\rho, \cdot) = 
 \rho^*\bar{\om} (\tb; \rho_0, \ldots , \rho_p, \cdot) ,
\quad \forall \, \rho, \rho_i \in \Gb ,
\end{align*}
is made via the exchange relations
 \begin{align*} 
 \begin{split}
  &\om(\tb; \phi_1, \ldots , \phi_p, x) = 
  \bar{\om} (\phi_1 \cdots \phi_p ,\,  \phi_2 \cdots \phi_p ,\,
  \ldots , \phi_p, x) \, , \\
  \text{resp.} \quad 
 &\bar{\om}(\tb; \rho_0, \ldots , \rho_p , \cdot) = 
 \rho_p^* \om(\tb; \rho_0 \rho_1^{-1}, \rho_1 \rho_2^{-1} ,
 \ldots , \rho_{p-1} \rho_p^{-1}, \cdot) .
 \end{split}
 \end{align*}
 
In~\cite{M-EC} we introduced the subcomplex  
$\{\Om_{\rm rd}^\bullet (|\bar\triangle_{\Gb} M|), d \}$ of the above complex
consisting of {\em regular differentiable} simplicial de Rham forms.
 These are compatible forms $\om = \{\om_p\}_{p\geq 0}$
on the geometric 
realization $|\triangle_{\Gb} M| = \prod_{p=0}^\ify \D^p \ts \triangle_{\Gb} M[p]$ of
the simplicial manifold $\triangle_{\Gb} M = \{ \Gb^p \ts M\}_{p\geq 0}$, whose
components (expressed in homogeneous group coordinates, but with the 
``overline'' mark omitted from the notation from now on) have the property that
 \begin{align} \label{diffo}
 \begin{split}
\om_p (\tb; \rho_0, \ldots , \rho_p, x) = 
\sum P_{I, J} \left(\tb; x, j^k_x(\rho_0), \ldots , j^k_x(\rho_p) \right) dt_I \wg dx_J ,
 \end{split}
\end{align}
with $P_{I, J}$ depending polynomially of a finite number of
jet components of $\rho_a$,  $1 \leq a \leq p$ and of $\big(\det \rho'_a (x)\big)^{-1}$,
where $\rho'_a(x)$ signifies the Jacobian matrix 
$ \rho'_a(x)_i^j = \p_i\rho^j_a (x)$, $1\leq i,j \leq n$.
 
 $ \Om_{\rm rd}^{\bullet} (|\triangle_{\Gb}M|)$ is a differential graded
 algebra, whose corresponding cohomology ring $H_{\rm rd}^\bullet  (|\bar\triangle_{\Gb} M|, \Cb)$ 
 was shown in~\cite[Thm. 1.4]{M-EC} to be
 isomorphic to the truncated polynomial ring of Chern classes 
 $P_{2n} [c_1, \ldots, c_n]$.

 More precisely, let  $\nb$ be the flat connection on the frame bundle
$FM \ra M$, with
 connection form \, $\om_\nb = \left(\om^i_j \right)$, 
$\,  {\om}^i_j \, :=  \, ({\bf y}^{-1})^i_{\lambda} \, d{\bf y}^{\lambda}_j = \big({\bf y}^{-1} \, d{\bf y}\big)^i_j$,
 \, $i, j =1, \ldots , n$. The associated simplicial connection form-valued matrix
   $\hat{\om}_\nb = \{ \hat\om_p \}_{p \in \Nb}$ on the frame bundle of  $|\triangle_{\Gb} M|$
   has components
 \begin{align} \label{scone}
 \hat\om_p (\tb ; \rho_0, \ldots , \rho_p) : = \sum_{i=0}^p t_i \rho_i^* (\om_\nb) ;
\end{align}
accordingly, the 
 simplicial curvature form 
 $\hat{\Om}_\nb : = d \hat{\om}_\nb + \hat{\om}_\nb \wg \hat{\om}_\nb$
has components  
\begin{align}  \label{scurv}
\begin{split}
\hat{\Om}_p (\tb ; \rho_0, \ldots , \rho_p)& = \, \sum_{i=0}^p dt_i \wdg \rho_i^* (\om_\nb) \, -
\, \sum_{i=0}^p t_i   \rho_i^* (\om_\nb) \wdg \rho_i^* (\om_\nb) \\
&+ \, \sum_{i, j=0}^p t_i t_j \,  \rho_i^* (\om_\nb) \wdg  \rho_j^* (\om_\nb) .
\end{split}
 \end{align} 
Under the action of $\rho \in \Gb$ on $FM$ the pull-back of the connection form is 
\begin{align} \label{phicon} 
\rho^* ({\om}^i_j )&={\om}^i_j \, +\, \g_{jk}^i (\rho)  \,\theta^k  , \qquad 
 \text{where} \qquad \theta^k = ({\bf y}^{-1}\cdot dx)^k  \\ \notag
\g^i_{j \, k} (\rho) (x, {\bf y})&=\left( {\bf y}^{-1} \cdot
{\rho}^{\prime} (x)^{-1} \cdot \part_{\lambda} {\rho}^{\prime} (x) \cdot
{\bf y}\right)^i_j \, {\bf y}^{\lambda}_k ,\quad x \in M, \, \bf y \in \GL_n(\Rb) .
\end{align}
This clearly shows that the simplicial forms $\hat{\om}_j^i$ and $\hat{\Om}_j^i$ belong to the 
regular differentiable de Rham complex $\Om_{\rm rd}^\bullet (|\bar\triangle_{\Gb} FM|)$.

 The cohomology $H_{\rm rd}^\bullet (|\bar\triangle_{\Gb} M|)$ of the 
 Dupont complex for $M$
 was shown in \cite{M-EC} to be isomorphic to the truncated polynomial ring of Chern classes
 $P_{2n}[c_1, \ldots , c_n]$. More precisely, by~\cite[Thms. 1.3 and Eq. (1.13)]{M-EC},
 the choice of the connection gives rise to a
 canonical quasi-isomorphism of complexes
\begin{align} \label{ChernWeil}
 \Cc^{\GL_n}_\nabla : \hat{W}(\Fg\Fl_n,  \GL_n) \ra 
\Om_{\rm rd}^\bullet (|\bar\triangle_{\Gb} M|) ,
\end{align}
which in fact reproduces the classical Chern-Weil construction for the $\Diff$-equivariant case.
Indeed, the left hand stands for the subalgebras consisting of the $ \GL_n$-basic elements
in  the quotient $\,\hat{W}(\Fg\Fl_n) = W(\Fg\Fl_n)/ \Ic_{2n}$ of the Weil
algebra $\, W(\Fg\Fl_n) = \wg^\bullet \Fg\Fl^*_n \ot S(\Fg\Fl_n)$ by the ideal
generated by the elements of $S(\Fg\Fl_n) $ of degree $> 2n$. The cohomology of 
$\hat{W}(\Fg\Fl_n,  \GL_n)$ is well-known to be
isomorphic to $P_{2n} [c_1, \ldots, c_n]$, 
 with $c_1, \ldots, c_n$ given by the standard generators of the ring $S(\Fg\Fl_n)^{\GL_n} $ of
$\GL_n (\Cb)$-invariant polynomials on $\Fg\Fl_n (\Cb)$,
\begin{align} \label{uchern}
 c_k (A) \, = \, \sum_{1\le i_1 < \ldots < i_k \le n} \sum_{\lambda \in S_k} 
 (-1)^\lambda A^{i_1}_{\lambda(i_1)} \cdots A^{i_k}_{\lambda(i_k)} ,
  \quad A \in \Fg \Fl_n (\Cb).
\end{align}
The corresponding Chern forms 
$c_k (\hat\Om_\nb) \in \Om_{\rm rd}^\bullet (|\bar\triangle_{\Gb} FM|)$ 
are $\GL_n$-basic and thus
descend to forms in $ \Om_{\rm rd}^{2k} (|\triangle_{\Gb}M|)$. 
As a result, the collection of closed forms
\begin{align} \label{chern-forms}
  c_J (\hat{\Om}_\nb) \, = \, 
   c_{j_1}(\hat{\Om}_\nb) \wg  \ldots \wg c_{j_q} (\hat{\Om}_\nb)
  \in \Om_{\rm rd}^{2|J|} (|\triangle_{\Gb}M| ,
\end{align}
with $J= (j_1 \leq \ldots \leq j_q)$ and $|J| := j_1 +\ldots + j_q \leq n$, 
give a complete set of representatives for a (linear) 
basis of the cohomology $H_{\rm rd}^\bullet  (|\bar\triangle_{\Gb} M|, \Cb)$
of the Rham complex $\{\Om_{\rm rd}^\bullet (|\triangle_{\Gb} M|), d \}$.
\smallskip

  Consider now the subcomplex $\{\bar{C}_{\rm rd}^\bullet \left(\Gb, \Om^\bullet (M)\right), \d, d \}$
  of the (homogeneous version of the) Bott bicomplex (see~\cite{Bott*, BSS}) 
  $\{\bar{C}^\bullet \left(\Gb, \Om^\bullet (M)\right), \d, d \}$,
  formed of {\em regular differentiable} homogeneous group cochains. 
  By definition (\cf~\cite{M-EC}), a cochain $\om \in \bar{C}_{\rm rd}^p \left(\Gb, \Om^q (M)\right)$
if for any local chart $U \subset M$ with coordinates 
$(x^1, \ldots , x^n)$,
\begin{align} \label{difco}
\om (\rho_0, \ldots , \rho_p, x) = 
\sum P_I \left(x, j^k_x(\rho_0), \ldots , j^k_x(\rho_p) \right) dx_I ,
\end{align}
where the coefficients $P_I$  as in \eqref{diffo}. By~\cite[Thm. 1.1]{M-EC}, 
which is a variant of Dupont's~\cite[Theorem 2.3]{Dupont}, the operation of
integration along the fiber
\begin{align} \label{circint}
 \oint_{\D^\bullet} : \Om_{\rm rd}^\bullet (|\bar\triangle_{\Gb} M|) \ra 
\bar{C}_{\rm rd}^{\bullet} \left(\Gb, \Om^\ast (M)\right)
\end{align}
establishes a quasi-isomorphism
between the complexes  $\{\Om_{\rm rd}^\bullet (|\bar\triangle_{\Gb} M|), d \}$ and
$\{\bar{C}_{\rm rd}^{\rm tot} \left(\Gb, \Om^\ast (M)\right), \d \pm d \}$. 
Thus, the composition of \eqref{ChernWeil} and \eqref{circint}
\begin{align} \label{ChernWeil2}   
 \Dc^{\GL_n}_\nabla :=  \oint_{\D^\bullet} \circ  \Cc^{\GL_n}_\nabla 
 : \hat{W}(\Fg\Fl_n,  \GL_n) \ra 
\bar{C}_{\rm rd}^{\rm tot \, \bullet} \left(\Gb, \Om^\ast (M)\right)
\end{align} 
 is also a quasi-isomorphism. In conclusion the cocycles
 \begin{align} \label{basis-forms}
  C_J (\hat{\Om}_\nb) \, := \, \oint_{\D^\bullet} c_J (\hat\Om_\nb) ,
  \quad  J= (j_1 \leq \ldots \leq j_q),\, \,  |J| \leq n  \} .
  \end{align}
represent a basis for the cohomology $H_{\rm rd, \Gb}^\bullet  (M, \Cb)$
of $\{ \bar{C}_{\rm rd}^{\rm tot} \left(\Gb, \Om^\ast (M)\right), \d \pm d \}$.
 
\medskip

On the other hand, in~\cite{mr11} we have introduced 
 Hopf cyclic counterparts of the Bott complexes. In particular,
 in the case of $\Fc_\Kc$ 
the analogue of the homogeneous Bott complex
is the anti-symmetrized and coinvariant subcomplex of  \eqref{bicomp-V*},
\begin{equation}\label{coinv}
\bar{C}^{\bullet} (\wg V^\ast, \wg\Fc_\Kc)\, = (
\wg^pV^\ast\ot \wdg\Fc_\Kc)^{\Fc_\Kc}  , \quad V\equiv M = \Rb^n ,
\end{equation}
defined as follows. 
An element  $ \sum \a\ot \td f\in ( \wg^pV^\ast\ot  \wdg^{q+1}\Fc_\Kc)^{\Fc_\Kc} $
if it satisfies the $\Fc_\Kc$-coinvariance condition:
\begin{equation}\label{coinv-condition}
\sum \a\ns{0}\ot \;\td{f}\ot S(\a\ns{1})=\sum \; \a\ot\td{f}\ns{0}\ot \td{f}\ns{1} ;
\end{equation}
here for  $\td f=f^0\wdots f^q$, we have denoted
\begin{align} \label{coinv-not}
\td f\ns{0}\ot \td f\ns{1}\, = \, f^0\ps{1}\wdots
f^q\ps{1}\ot f^0\ps{2}\cdots f^q\ps{2}.
\end{align}
One  identifies  the anti-symmetrized-coinvariant   bicomplex as a  homotopy retraction sub-bicomplex of   \eqref{bicomp-V*} as in \cite{mr11}. 
\begin{align}\label{wedge-coinv-FK}
\begin{xy} \xymatrix{  \vdots & \vdots
 &\vdots &&\\
 \wdg^2V^\ast  \ar[u]^{\p_{{\wg}}}\ar[r]^{b_{{\wg}}~~~~~~~}&  (\wdg^2V^\ast\ot\wdg^2 \Fc_\Kc)^{\Fc_\Kc} \ar[u]^{\p_{{\wg}}} \ar[r]^{b_{{\wg}}}& (\wdg^2V^\ast\ot\wdg^3\Fc_\Kc)^{\Fc_\Kc} \ar[u]^{\p_{{\wg}}} \ar[r]^{~~~~~~~~~b_{{\wg}}} & \hdots&  \\
 V^\ast  \ar[u]^{\p_{{\wg}}}\ar[r]^{b_{{\wg}}~~~~~}&  (V^\ast\ot\wdg^2 \Fc_\Kc)^{\Fc_\Kc} \ar[u]^{\p_{{\wg}}} 
 \ar[r]^{b_{{\wg}}}& (V^\ast\ot\wdg^3\Fc_\Kc)^{\Fc_\Kc} \ar[u]^{\p_{{\wg}}} \ar[r]^{~~~~~b_{{\wg}} }& \hdots&  \\
 \Cb  \ar[u]^{\p_{{\wg}}}\ar[r]^{b_{{\wg}}~~~~~~~}&  (\Cb\ot\wdg^2 \Fc_\Kc)^{\Fc_\Kc} \ar[u]^{\p_{{\wg}}} \ar[r]^{b_{{\wg}}}& (\Cb\ot\wdg^3\Fc_\Kc)^{\Fc_\Kc} \ar[u]^{\p_{{\wg}}} \ar[r]^{~~~~~b_{{\wg}}} & \hdots&,  }
\end{xy}
\end{align}
The identification simplifies the action of $V$ on $\bigwedge \Fc_\Kc$  into the diagonal action
\begin{align} \label{diagact}
X\rt (f^0\odots f^q )=\sum_{i=0}^q f^0\odots X\rt f^i\odots f^q  .
\end{align}
 and also the coboundaries are simplified to
\begin{align}\label{b-F-wedge}
b_{{\wg}}(\a\ot f^0\wdots f^q)= \a\ot 1\wdg f^0\wdots f^q,
\end{align}
and
\begin{align}
\begin{split}
&\p_{{\wg}}(\a\ot f^0\wdots f^q)= -\, \sum_i \t^i\wdg \a\ot \wedge\ot X_i\rt(f^0\wdots f^q).
\end{split}
\end{align}
 The total cohomology of this bicomplex is denoted by 
 $HP_{\rm CE}^\bullet (\Kc_n, ^{\s^{-1}}\Cb)$.
\smallskip

A similar bicomplex is defined for $\Fc_\Hc)$, namely 
\begin{equation} \label{wedge-coinv-FH}
\bar{C}^{\bullet} (\wg V^\ast, \wg\Fc_\Hc)^{\Fg\Fl_n}:=
\bigl((\wg^pV^\ast\ot \wdg^{q+1} \Fc_\Hc)^{\Fc_\Hc}\bigr)^{\Fg\Fl_n} ,
\end{equation}
and the restriction map $r_\Hc:\Fc_\Kc\ra \Fc_\Hc$
induces a  quasi-isomorphism between  \eqref{wedge-coinv-FK} and \eqref{wedge-coinv-FH}. 
 
 The total cohomology of this bicomplex is denoted by $HP_{\rm CE}^\bullet (\Hc_n, \GL_n : \Cb_\d)$
\smallskip

At this stage we recall that in~\cite[\S 3.2]{mr11}  we have constructed a map 
of bicomplexes, $\Theta$ from the bicomplex
 $\bar{C}^{\bullet} (\wg \Fg_{\rm aff}^\ast, \wg\Fc_\Hc)^{\Fg\Fl_n}$ of antisymmetrized
 $\Fc_\Hc$-coinvariant cochains  
 to $\bar{C}^{\bullet}(\Gb, \Om^*(FM))$. 
 In order to write its expression, we also need to recall (\cf \cite{mr09}) that
there is a canonical isomorphism $\etabar: \Hc^{\cop}_{\rm ab} \rightarrow \Fc_\Hc$ 
and $\, \dbar = \etabar^{-1}$ denotes its inverse. For $f \in \Fc_\Hc$, one defines the function
$\gbar_\Hc(f) :\Gb \ra C^\ify (G)$ by
 \begin{align} \label{gbar}
 \dbar(S(f))(U_\phi) \, = \,  \gbar_\Hc(f)(\phi) \,U^*_\phi ,  \qquad \fl \, \phi \in \Gb ,
\end{align}
where the left hand side uses the action of $\Hc_n$ on the crossed product algebra
$C^\ify (G_{\rm aff} ) \rtimes \Gb$, with 
$G_{\rm aff} =V \ltimes \GL_n(\Rb)$.  The function $ \gbar_\Hc(f)(\phi) \in C^\ify (G) $,
depends algebraically on the components of the $k$-jet 
of $\phi$, for some $k \in \Nb$. For example, if 
$f= \eta^i_{j k \ell_1 \ldots \ell_r}$ is one of the canonical algebra generators of $\Fc_\Hc$,
\begin{equation} \label{fcoord}
 \eta^i_{j k \ell_1 \ldots \ell_r} (\psi) =  \g^i_{j k \ell_1 \ldots
\ell_r} (\psi)(e) ,  \quad \psi \in \Nbo
 \end{equation}
then
\begin{align}  \label{bargeta2}
\gbar_\Hc(S(\eta^i_{j k \ell_1 \ldots \ell_r}))(\psi^{-1})\, = \, \g^i_{j k \ell_1 \ldots \ell_r} (\psi) .
\end{align}

With these notions clarified, the map $\Theta$ is given by the formula
\begin{align} \label{Theta}
\begin{split}
\Theta&\big(\sum_{I} \a_I\ot \lu{I}f^0 \wg \cdots
 \wg \lu{I}f^p \big)(\phi_0, \dots ,\phi_p)=\\
&\sum_{I} \sum_{\s \in S_{p+1}} 
(-1)^\s   \gbar_\Hc(S( \lu{I}f^{\s(0)}))(\phi_0^{-1})\dots \gbar_\Hc(S(\lu{I}f^{\s(p)}))(\phi_p^{-1}){{\td\a_I}}  ,
\end{split}
\end{align}
where $\Fg_{\rm aff}$ is the Lie algebra of the affine group $G_{\rm aff} \cong FM$, and 
$\{\td\a_I\}$ are the left-invariant form associated to the elements of a basis 
$\{\a_I\} \subset \wg^\bullet \Fg_{\rm aff}^*$.  

From its very definition, it is obvious that $\Theta$ actually lands in 
$\bar{C}_{\rm rd}^{\bullet}(\Gb, \Om^*(FM))$. It is also transparent that
 $\Theta$ is injective.

On the other hand, $\Theta$ is clearly $\GL_n$-equivariant and thus, by restriction to
$\GL_n$-invariants, it gives the map
$\Theta^{\GL_n}: \bar{C}^{\bullet} (\wg V^\ast, \wg\Fc_\Hc)^{\Fg\Fl_n}
 \ra \bar{C}_{\rm rd}^{\bullet}(\Gb, \Om^*(M))$,
 \begin{align} \label{Theta-rel}
\begin{split}
\Theta^{\GL_n}&\big(\sum_{|I|=q} dx_I\ot \lu{I}f^0 \wg \cdots
 \wg \lu{I}f^p \big)(\phi_0, \dots ,\phi_p)=\\
&\sum_{I} \sum_{\s \in S_{p+1}} 
(-1)^\s   \gbar_\Hc(S( \lu{I}f^{\s(0)}))(\phi_0^{-1})\dots \gbar_\Hc(S(\lu{I}f^{\s(p)}))(\phi_p^{-1}){dx_I}  .
\end{split}
\end{align}

\begin{theorem} \label{vErd}
 The chain map  
 $\Theta^{\GL_n}: \bar{C}^{\bullet} (\wg V^\ast, \wg\Fc_\Hc)^{\Fg\Fl_n} \ra
  \bar{C}_{\rm rd}^{\bullet}(\Gb, \Om^*(M))$
 is a quasi-isomorphism. 
 \end{theorem}
 
 \begin{proof}
The proof follows along exactly the same lines as that of \cite[Thm. 3.6]{M-GC},
the only difference being that all the {\em differentiable subcomplexes} are replaced
by their {\em regular differentiable} counterparts. This also entails replacing the horizontal
homotopy used in the proof of \cite[Thm. 1.2]{M-GC} by the algebraic homotopy 
employed in the proof of \cite[Thm. 1.3]{M-EC}.  
 \end{proof}

The preimage by  $\Theta^{\GL_n}$ of the cocycles  $C_J (\hat{\Om}_\nb) $ 
defined by \eqref{basis-forms} can be exactly computed. Indeed, first we observe
that by~\cite[Remark 3.9]{mr11} the map $\Theta$ is insensitive to
affine transformations, \ie if  $\vp_0,\dots,\vp_q \in G_{\rm aff}$ and
 $\psi_0,\dots,\psi_q \in \Nbo$, then
 \begin{align} \label{no-aff}
\begin{split}
\Theta \bigl(\sum_I\a_I&\ot \lu{I}f_0 \wdg \cdots \wdg \lu{I}f_q \bigr)
(\vp_0\psi_0,\dots, \vp_q\psi_q) \, = \, \\
&\Theta \bigl(\sum_I\a_I\ot \lu{I}f_0 \wdg \cdots \wdg \lu{I}f_q \bigr)(\psi_0,\dots,\psi_q) .
\end{split}
\end{align}
Next we note that being given by invariant polynomials, 
 the Chern cocycles \eqref{basis-forms} are built out of the pull-back of the curvature
form by the cross-section $x \in \Rb^n \mapsto (x, {\bf 1})\in \Rb^n \times\GL_n$.
The resulting simplicial matrix-valued form is
\begin{align} \label{R-forms}
\begin{split}
 \hat{R}&(\tb ; \phi_0, \ldots , \phi_p) \, = 
 \sum_{r=0}^p dt_r \wdg  \G (\phi_r)
 - \sum_{r=0}^p t_r\,
\G(\phi_r)\wdg \G (\phi_r)  \\ 
&+ \sum_{r, s=0}^p t_r t_s \, \G (\phi_r)\wdg \G(\phi_s)  ,  \qquad \text{where} \quad
\G (\phi):= ({\phi}^{\prime})^{-1} \cdot d{\phi}^{\prime} ,
\end{split}
 \end{align}
which by restriction to $|\bar\triangle_\Nbo M|$ becomes
\begin{align} \label{R-forms2}
 \hat{R}(\tb ; \psi_0, \ldots , \psi_p)  = 
 \sum_{r=0}^p dt_r \wdg  d \psi_r'
 - \sum_{r=0}^p t_r\,d \psi_r' \wdg  d \psi_r'  
+ \sum_{r, s=0}^p t_r t_s \,  d \psi_r'\wdg  d \psi_s' .
 \end{align}
 For $\psi \in \Nbo$, 
$  (d \psi')^i_j \, = \, \sum_{k=1}^n \p_k \p_j \psi^i \, dx^k$
and so  
$ (d \psi')^i_j \mid_{x=0}\, = \, \sum_{k=1}^n \eta_{j k}^i (\psi)\, dx^k$.
This clearly shows that the restriction of the simplicial Chern form  $c_J (\hat{\Om}_\nb) $
to $|\bar\triangle_\Nbo M|$ evaluated at $x=0$  gives by integration over the simplices 
a cocycle $C_J( \hat{R}|_0)  \in C^{\bullet}_\Fc (\wg M^\ast, \wg\Fc_\Hc) $. Moreover, by
the very construction,
\begin{align} \label{preimH}
\Theta^{\GL_n}\bigl(C_J( \hat{R}|_0)\bigr) \, = \, C_J (\hat{\Om}_\nb) .
\end{align}
Combining Theorem \ref{vErd} with the statement \eqref{basis-forms} we obtain:
 
\begin{corollary}  \label{prebasisH}
The cocycles  $C_J( \hat{R}|_0)  \in \bar{C}^{\bullet}(\wg V^\ast, \wg\Fc_\Hc)^{\Fg\Fl_n} $, 
with $ J= (j_1 \leq \ldots \leq j_q)$, \, $|J| \leq n$, 
represent classes which form a basis for the cohomology
$HP_{\rm CE}^\bullet (\Hc_n, \GL_n : \Cb_\d)$.
\end{corollary}

\medskip

To reproduce the same approach for the algebra $\Kc_n$, we replace the map $\Theta$ by
its counterpart corresponding to the decomposition
$\Gb= \Tbo\cdot \Gb^\dagger$. The new chain map $\Theta_\Kc$, from the subcomplex
 $\bar{C}^{\bullet} (\wg V^\ast, \wg\Fc_\Kc)$ of 
 $\Fc_\Kc$-coinvariant cochains in $C^{\bullet}(\wg {V}^\ast, \wg\Fc_\Kc)$
 to $\bar{C}^{\bullet}(\Gb, \Om^*(FM))$, is defined by the similar formula
\begin{align} \label{ThetaK}
\begin{split}
\Theta_\Kc&\big(\sum_{I} \a_I\ot \lu{I}f^0 \wg \cdots
 \wg \lu{I}f^p \big)(\phi_0, \dots ,\phi_p)=\\
&\sum_{I} \sum_{\s \in S_{p+1}} 
(-1)^\s   \gbar_\Kc(S( \lu{I}f^{\s(0)}))(\phi_0^{-1})\dots \gbar_\Kc(S(\lu{I}f^{\s(p)}))(\phi_p^{-1}){{\td\a_I}} .  
\end{split}
\end{align}
The analogous property to \eqref{no-aff} reads as follows:
 if  $\vp_0,\dots,\vp_q \in \Tbo$ and
 $\psi_0,\dots,\psi_q \in \Gb^\dagger$, then
 \begin{align} \label{no-trans}
\begin{split}
\Theta_\Kc \bigl(\sum_I\a_I&\ot \lu{I}f_0 \wdg \cdots \wdg \lu{I}f_q \bigr)
(\vp_0\psi_0,\dots, \vp_q\psi_q) \, = \, \\
&\Theta_\Kc \bigl(\sum_I\a_I\ot \lu{I}f_0 \wdg \cdots \wdg \lu{I}f_q \bigr)(\psi_0,\dots,\psi_q) .
\end{split}
\end{align}
This follows from the simple fact that for any translation $\vp(x) = x + b$, one has
$\s^i_j (U_\vp^*) = \d^i_j U_\vp^*$.

The dual to the projection map $\psi \in \Gb^\dagger \mapsto \nu_\psi \in \Nbo$ 
gives an inclusion $\iota_\Hc : \Fc_\Hc \ra \Fc_\Kc$, which is defined by 
\[ \i_\Hc(f)=\Phi^{-1}(1\ot f),\]
where $\Phi^{-1}$ is defined in \eqref{PHI}.
One observes that $\i_\Hc$ is a cross-section of

the restriction map $r_\Hc:\Fc_\Kc \ra \Fc_\Hc$. In turn, $\iota_\Hc$
gives rise to a chain map $\iota_\Hc^\bullet : \bar{C}^{\bullet} (\wg V^\ast, \wg\Fc_\Hc)
\ra \bar{C}^{\bullet} (\wg V^\ast, \wg\Fc_\Kc)$ at the level of Hochschild complexes. Manifestly, one has
\begin{align} \label{Theta-comp}
 \Theta_\Kc \circ  \i_\Hc^{\bullet}\, = \, \Theta^{\GL_n}.
\end{align}
 
 \begin{lemma} \label{quas-iota}
 The chain map $\iota_\Hc^\bullet : \bar{C}^{\bullet} (\wg V^\ast, \wg\Fc_\Hc)^{\Fg\Fl_n}
\ra \bar{C}^{\bullet} (\wg V^\ast, \wg\Fc_\Kc)$ is a quasi-isomorphism of
bicomplexes.
 \end{lemma}
 
 \proof By construction, $r_\Hc^\bullet \circ \iota_\Hc^\bullet = \Id$. On the
other hand the very same arguments
invoked in the proof of Theorem \ref{thm-alt-chern} show that $r_\Hc^\bullet $
is a quasi-isomorphism in Hochschild cohomology, which moreover
induces an isomorphism in cyclic cohomology. Therefore, its right inverse
$\iota_\Hc^\bullet$, which is a chain map of bicomplexes, gives an
isomorphism in the cohomology of the total complexes.
\endproof

We next build the preimage  by  $\Theta_\Kc$ of the cocycles  $C_J (\hat{\Om}_\nb) $ 
in exactly the same fashion as for $\Theta^{\GL_n}$, except that instead of using the
simplicial curvature \eqref{R-forms2} on $|\D_\Nbo M|$, we use the simplicial
curvature form  \eqref{R-forms} on $|\D_{\Gb^\dagger} M|$. Note that the latter involves
the forms $\G (\psi):= ({\psi}^{\prime})^{-1} \cdot d{\psi}^{\prime}$, for $\psi \in \Gb^\dagger$.
The cocycles thus obtained, $C^\dagger_J( \hat{R}|_0)$ are uniquely determined by
the equation
\begin{align} \label{preimK}
\Theta_\Kc\bigl(C^\dagger_J( \hat{R}|_0)\bigr) \, = \, C_J (\hat{\Om}_\nb) .
\end{align}

\begin{corollary}  \label{prebasisK}
The cocycles  $C^\dagger_J( \hat{R}|_0)  \in \bar{C}^{\bullet}(\wg V^\ast, \wg\Fc_\Kc) $, with
$ J= (j_1 \leq \ldots \leq j_q)$, \, $|J| \leq n$, 
represent a basis of cohomology classes for
$HP_{\rm CE}^\bullet (\Kc_n ; \, ^{\s^{-1}}\Cb)$.
\end{corollary}

\proof The relation \eqref{Theta-comp} implies that $\iota_\Hc (C_J( \hat{R}|_0)) 
= C^\dagger_J( \hat{R}|_0)$, since $\Theta_\Kc$ is injective. The claim then 
follows from Lemma \ref{quas-iota}.
\endproof

\section{Characteristic map and Hopf cyclic Chern cocycles} \label{S5}

The crossed product algebra $\Ac = C_c^\infty (M) \rtimes \Gb$ has a canonical state-like
functional $\tau:\Ac\ra \Cb$, determined by the standard volume form on $M=\Rb^n$,
$\varpi = dx^1 \wg \ldots \wg dx^n$; it is given by
\begin{align}
\begin{split}
&\tau(fU^\ast_\phi)= \left\{\begin{matrix}\displaystyle   \int_{{M}} f\varpi,& \text{if} \quad \phi=\Id \\ &&\\
0, &\text{otherwise}\
\end{matrix}\right.
\end{split}
\end{align} 
Unlike its forerunner on the frame bundle employed in~\cite{cm2, cm3}, the 
linear map $\tau$ is not a trace. It is however easy to check that $\tau$ is a
$\s^{-1}$-trace, \ie
\begin{align}\label{sigma-trace}
\tau(ab)=\tau(b\s^{-1}(a)) , \qquad \forall \, a, b \in \Ac ,
\end{align}
and that it is $\ve $-invariant with respect to the action of $\Kc_n$, meaning that
 \begin{align} \label{delta-invariant}
\tau(k(a))=\ve(k)\tau(a) , \qquad \forall \, k \in \Kc_n, \, a \in \Ac ;
\end{align}
in particular, $\tau(\s(a))=\tau(a)$.

Having these two properties, one can define (\cf \cite{cm2, cm4}) 
a characteristic map $\chi_\tau$ from the standard Hopf cyclic $(b, B)$-complex
$CC^\bullet(\Kc_n; \,^{\s^{-1}}\Cb)$ to the cyclic cohomology $(b, B)$-complex
$ CC^\bullet(\Ac)$, by
\begin{align} \label{char-map}
 & \chi_\tau (k^1, \ldots,k^q)(a_0\ldots,a_q )=\tau(a_0k^1(a^1)\cdots k^q(a_q)), \\
 &\qquad k^1, \ldots,k^q \in \Kc_n , \qquad  a_0, \ldots, a_q \in \Ac .
\end{align}
which is a map of cyclic complexes. As a matter of fact, this map is injective
and the model for the Hopf cyclic structure in the left hand side was originally 
imported in~\cite{cm2} from that of the right hand side. We will show below 
that this structural characteristic map also allows to transfer the geometric
cocycles constructed in \S \ref{S4} to the Hopf cyclic 
complex $CC^\bullet(\Kc_n ; \,^{\s^{-1}}\Cb)$.
\medskip

Connes has constructed  (see~\cite[III.2.$\d$]{NCGbook}) a map of
bicomplexes
\begin{align*}
\Phi_{\rm C}:  \bar{C}^{\bullet}(\G, \Om^\bullet(M)) \ra CC^\bullet (C_c^\infty (M) \rtimes \Gb) ,
\end{align*}
whose definition we quickly recall.

Let ${\Bc_\Gb} (M)$ denote the DG-algebra 
 $\Om^*_c(M) \ot \wg \, \Cb [\Gb']$, where $\Gb' = \Gb \setminus \{e\}$
 with the differential $d \ot \Id$.
After labeling the generators of $\Cb [\Gb']$ as
$\g_{\phi}$, $\phi \in \Gb$, with $\g_1 = 0$,  one forms the crossed product
$\, \Cc_\Gb (M)  = {\Bc_\Gb}(M) \rtimes \Gb$, 
with the multiplication rules
\begin{align*}
&U_{\phi}^\ast \, \om \, U_{\phi} = \phi^\ast \, \om  , &\qquad
\, \om \in \Om^*_c(M),\\
& U_{\phi_1}^\ast \, \g_{\phi_2} \, U_{\phi_1} =\g_{\phi_2 \circ \phi_1} -
\g_{\phi_1} , &\qquad  \phi_1 , \phi_2 \in \Gb \, .
\end{align*}
$\Cc_\Gb (M) $ is itself a DG-algebra, equipped with the differential  
\begin{equation} \label{dbig}
d (b \, U_{\phi}^\ast) = db \, U_{\phi}^\ast - (-1)^{\p b} \, b \, \g_{\phi} \,
U_{\phi}^\ast  , \qquad b \in {\Bc_\Gb} (G) , \quad \phi \in \Gb,
\end{equation}
Any $\lambda \in \bar{C}^{q}(\Gb, \Om^p(M))$ gives rise a linear form
$\wt{\lambda}$ on $\Cc_\Gb (G) $ as follows: 
\begin{align}   \label{prePhi}
\begin{split}
&\wt{\lambda} (b \,U_{\phi}^\ast) = 0 \qquad  \text{for} \quad  \phi \ne 1 ; \\
& \text{if}  \quad \phi = 1 \quad  \text{and} \quad
 b=\om \ot \g_{\rho_1} \ldots \g_{\rho_q} \qquad  \text{then} \\
&\wt{\lambda}(\om \ot \g_{\rho_1} \ldots \g_{\rho_q}) = \int_{M}
 \lambda(1, \rho_1 , \ldots ,\rho_q) \wg \om .
   \end{split}
\end{align}
The map $\Phi_{\rm C}$ from $ \bar{C}^{\bullet}(\Gb, \Om^\bullet(G))$
to the $(b, B)$-complex of the algebra $\Ac = C_c^\ify (M)  \rtimes \Gb$ 
is now defined for  $\lambda \in \bar{C}^{q}(\Gb, \Om^p(M))$ by
\begin{align}   \label{mapPhi}
\begin{split}
\Phi_{\rm C}(\lambda)(a^0, \ldots, a^m)&=
\frac{p!}{(m+1)!}
\sum_{j=0}^m (-1)^{j(m-j)}\wt{\lambda}(da^{j+1}\cdots da^m\; a^0\; 
da^1\cdots da^j) \\
   \text{where} \quad m &=\dim G-p+q  , \qquad a^0, \ldots, a^m \in \Ac .
   \end{split}
\end{align}
As proved in~\cite[III.2, Thm. 14]{NCGbook}, $\Phi_{\rm C}$ is a chain map to the total
 $(b, B)$-complex of the algebra $\Ac$.

We denote by $\Phi_{\rm rd}$ the restriction of $\Phi_{\rm C}$ to the subcomplex 
\begin{align*} 
\bar{C}_{\Theta}^{\rm tot}(\Gb, \Om^*(M)) :=
\Theta_\Kc \big(\bar{C}^{\bullet} (\wg V^\ast, \wg\Fc_\Kc)\big) 
\subset \bar{C}_{\rm rd}^{\rm tot}(\Gb, \Om^*(M)) .
\end{align*}
By reasoning as in~\cite[pp 223-234]{cm2},  it can be shown that
 if $\lambda \in \bar{C}_{\Theta}^{q}(\Gb, \Om^p(M))$  then there exists
 $\td{k}(\lambda) = \sum_\a  \dot k_\a^1 \ot \ldots \ot \dot k_\a^q \in \Kc_n^{\ot \, q}$
 such that
\begin{align*} 
 \Phi_{\rm C} (\lambda) \, = \, \sum_\a \chi_\tau (k_\a^1, \ldots, k_\a^q) ;
 \end{align*}
due to the faithfulness of $\chi_\tau$, the element $\td{k}(\lambda)$
 is necessarily unique. This gives a canonical identification between the two
 $(b, B)$-complexes, 
 \begin{align} \label{2bB}
 CC^\bullet(\Kc_n; \,^{\s^{-1}}\Cb) \cong  \Im (\Phi_{\rm rd}) ,
 \end{align}
 which allows us to regard  $\Phi_{\rm rd}$ as a chain map to $ CC^{\rm tot} (\Kc_n\,^{\s^{-1}}\Cb)$.
  
  \begin{theorem} \label{chernHP}
 The map $\Phi_{\rm rd}: \bar{C}_{\Theta}^{\rm tot}(\Gb, \Om^*(M)) \ra 
 CC^{\rm tot} (\Kc_n ;\,^{\s^{-1}}\Cb)$ is a quasi-isomorphism. Moreover, via
 the above identification, the cocycles 
\begin{align} \label{basis2}
\kappa_J (\hat{\Om}_\nb) := \Phi_{\rm rd}\bigl(C_J (\hat{\Om}_\nb)\bigr) \in 
 CC^{\rm tot} (\Kc_n\,^{\s^{-1}}\Cb) ,
\end{align}
 with $ J= (j_1 \leq \ldots \leq j_q)$, \, $|J| \leq n$, 
represent a basis of cohomology classes for $HP^\bullet (\Kc_n; \,^{\s^{-1}}\Cb)$.
 \end{theorem}
 
 \proof  By construction, 
\begin{align*}
 \Phi_{\rm rd} \circ \Theta_\Kc \circ  \i_\Hc^{\bullet} = \Phi_{\rm rd} \circ \Theta^{\GL_n} .
 \end{align*}
 The right hand side was shown to be a quasi-isomorphism in \cite[\S 2.2]{M-EC},
 while $ \i_\Hc^{\bullet}$ is quasi-isomorphism by Lemma \ref{quas-iota}. 
\endproof

\section{Explicit calculations for $n=1$}

We illustrate the above results, by producing completely explicit cocycles for the 
Hopf cyclic classes of  $\Kc_1$ and $(\Hc_1, \GL_1)$.

The connection form on $FM = \Rb \times \Rb^{\times}$ being 
$\, \om \equiv {\om}^1_1 \, :=  \, {\bf y}^{-1}\, d{\bf y}$, the
associated simplicial connection form is
 \begin{align*} 
 \hat\om_p (\tb ; \rho_0, \ldots , \rho_p) : = \sum_{i=0}^p t_i \rho_i^* (\om) =
 \sum_{i=1}^p s_i (\rho_{i-1}^* (\om) - \rho_i^* (\om)) + \rho_p^* (\om) .
\end{align*}
The pull-back of the connection form is 
\begin{align*}  
\rho^* (\om^1_1)&=\om^1_1 \, +\, \g_{11}^1 (\rho)  \,{\bf y}^{-1}\cdot dx  , \\ \notag
\g^1_{1\, 1} (\rho) (x, {\bf y})&= {\bf y}^{-1} \cdot
{\rho}^{\prime} (x)^{-1} \cdot \part {\rho}^{\prime} (x) \cdot {\bf y} \, {\bf y}  ,
\end{align*}
that is
\begin{align*}  
\rho^* (\om)\, =\,  {\bf y}^{-1} d{\bf y} \, + \, \frac{\rho ''}{\rho'}  dx  . 
\end{align*}
One has
 \begin{align*} 
 \hat\om_p (\tb ; \rho_0, \ldots , \rho_p) =
 \sum_{i=1}^p s_i \left(\frac{\rho''_{i-1}}{\rho'_{i-1}}  - 
 \frac{\rho''_{i}}{\rho'_{i}}\right) dx+ \frac{\rho''_p}{\rho'_p} dx +  {\bf y}^{-1} d{\bf y} .
\end{align*}
 The simplicial curvature form
 $\hat{\Om} = d \hat{\om}_\nb + \hat{\om}_\nb \wg \hat{\om}_\nb$
has components  
\begin{align*}  
\hat{\Om}_p (\tb ; \rho_0, \ldots , \rho_p)& =  \sum_{i=1}^p \left(\frac{\rho''_{i-1}}{\rho'_{i-1}}  - 
 \frac{\rho''_{i}}{\rho'_{i}} \right) ds_i dx  \\
&+\left( \sum_{i=1}^p s_i \left(\frac{\rho''_{i-1}}{\rho'_{i-1}}  - 
 \frac{\rho''_{i}}{\rho'_{i}}\right) +  \frac{\rho''_p}{\rho'_p}\right) {\bf y}^{-1} dx d{\bf y} .
 \end{align*} 
Its pull-back by the canonical section is  the curvature form
 \begin{align*}  
\hat{R}_p (\tb ; \rho_0, \ldots , \rho_p) =  \sum_{i=1}^p \left(\frac{\rho''_{i-1}}{\rho'_{i-1}}  - 
 \frac{\rho''_{i}}{\rho'_{i}} \right) ds_i dx  .
 \end{align*} 
 
The image $ \{ \oint_{\D_p} \hat{R}_p \}_p$ in the Bott complex is nontrivial only 
for  $p=1$, giving the cochain $ c_1 \in \bar{C}^1 (\Gb, \Om^1(M))$,
\begin{align}  \label{c1B}
c_1 (\rho_0, \rho_1) = \left(\frac{\rho''_0}{\rho'_0 } - \frac{\rho''_1}{\rho'_1}\right) dx
 \end{align} 
 Let us compute $\Phi_{\rm C} (c_1)$. Recall that $\td c_1$ is the current
 \begin{align}  
\td c_1 (\a \ot \g_\vp) = \int_M c_1(1, \vp) \wg \a , \quad a \in \Om_c^q (M) ,
 \end{align} 
which only pairs nontrivially if $\a \in \Om_c^0 (M)$. Then
 \begin{align}  \label{1c1}
\Phi_{\rm C}(c_1) (a_0, a_1) = \frac{1}{2}\td c_1 (da_1\, a_0 + a_0\, da_1) , \quad a_0, a_1 \in \Ac .
 \end{align} 
  Take $a_0 = f_0 U^*_{\rho_0}$ and $a_ 1= f_1 U^*_{\rho_1}$ with ${\rho_1} {\rho_0} = 1$.
Then $a:=a_0 \, a_1 =  f_0 \, \rho^*(f_1) := f$, hence 
 \begin{align}  \label{stokes}
 \td c_1 (da)\, = \, \td c_1 (df) \, = \,0.
 \end{align} 
So we can rewrite \eqref{1c1} as
 \begin{align}  \label{2c1}
\Phi_{\rm C}(c_1) (a_0, a_1) = \frac{1}{2}\td c_1 (da_1\, a_0 - da_0\,a_1) , \quad a_0, a_1 \in \Ac .
 \end{align} 
One has
 \begin{align*}  
 da_1\, a_0 &- da_0\, a_1= (df_1U^*_{\rho_1} - f_1 \g_{\rho_1}  U^*_{\rho_1}) f_0 U^*_{\rho_0}
 - (df_0U^*_{\rho_0} - f_0 \g_{\rho_0}  U^*_{\rho_0}) f_1 U^*_{\rho_1}\\
 &=df_1U^*_{\rho_1}f_0 U^*_{\rho_0} - f_1 \g_{\rho_1}  U^*_{\rho_1} f_0 U^*_{\rho_0}
 - df_0U^*_{\rho_0}  f_1 U^*_{\rho_1} + f_0 \g_{\rho_0}  U^*_{\rho_0} f_1 U^*_{\rho_1} \\
 &=df_1\, {\rho_1}^*( f_0) - f_1 \,{\rho_1}^*( f_0) \, \g_{\rho_1} 
- df_0 \,{\rho_0}^*( f_1) + f_0 \,{\rho_0}^*( f_1)\, \g_{\rho_0} .
 \end{align*} 
Hence
 \begin{align*}  
2\Phi_{\rm C}(c_1) (a_0, a_1)&= \td c_1 (f_0 \,{\rho_0}^*( f_1)\, \g_{\rho_0} - 
 f_1 \,{\rho_1}^*( f_0) \, \g_{\rho_1}) \\
&=\int_M f_0 \,{\rho_0}^*( f_1) c_1(1, \rho_0) - \int_M f_1 \,{\rho_1}^*( f_0) c_1(1, \rho_1) 
 \end{align*} 
So letting ${\rho_0} = \phi, {\rho_1} = \phi^{-1}$, one has
  \begin{align*}  
 2\Phi_{\rm C}(c_1) (a_0, a_1)&=\int_M f_0 \,{\rho_0}^*( f_1)  \left( - \frac{\rho''_0}{\rho'_0}\right) dx
  - \int_M f_1 \,{\rho_1}^*( f_0) \left(- \frac{\rho''_1}{\rho'_1}\right) dx \\
&= -\int_M f_0 \,(f_1\circ \phi) \, \frac{\phi''}{\phi'}\, dx
  + \int_M f_1 \,( f_0 \circ \phi^{-1}) \, \frac{({\phi}^{-1})''}{({\phi }^{-1})'}\, dx .
 \end{align*} 
Note now that, by substitution one has
 \begin{align*}  
& \int_M f_0(x) \,(f_1(\phi(x)) \, \frac{\phi''(x)}{\phi'(x)}\, dx = \int_M f_1(x) \,(f_0(\phi^{-1}(x))
\frac{\phi''(\phi^{-1}(x)}{\phi'(\phi^{-1}(x)}\, (\phi^{-1})'(x) dx \\
&=  \int_M f_1(x) \,(f_0(\phi^{-1}(x))
 \phi''(\phi^{-1}(x)\, (\phi^{-1})'(x)^2 dx  \\
&= - \int_M f_1(x) \,(f_0(\phi^{-1}(x)) \, \frac{(\phi^{-1})''(x)}{(\phi^{-1})'(x))} dx ;
 \end{align*} 
 the last equality uses the elementary identity
 \begin{align*}  
 \phi''(\phi^{-1}(x))\, (\phi^{-1})'(x)^2 \, + \, \frac{(\phi^{-1})''(x)}{(\phi^{-1})'(x))} \, = \, 0
 \end{align*} 
Thus we get
   \begin{align}  \label{pre-c1}
   \begin{split}
2 \Phi_{\rm C}&(c_1) (a_0, a_1) = 2\int_M f_1 \,( f_0 \circ \phi^{-1}) \, \frac{({\phi}^{-1})''}{({\phi}^{-1})'}\, dx \\
&=2 \tau(\s^{-1} \s^1_{11} (a_1) a_0) \, = \, 2 \tau(a_0 \, \s^{-2} \s^1_{11} (a_1)) .
\end{split}
\end{align}  
Equivalently,
  \begin{align}  \label{c1}
 \Phi_{\rm C}(c_1) (a_0, a_1)\, = \,  \chi_\tau (\s^{-2} \s^1_{11}) (a_0, a_1) .
\end{align} 
The other class arises from the constant simplicial form $\one \in \Om^0(|\D_{\Gb}M|)$,
which gives the cochain $c_0 \in \bar{C}^0(\Gb, \Om^0(M))$, 
\begin{align}  \label{c0B}
c_0(\rho) \equiv 1\, ; 
\end{align}
thus
$\td c_0$ is the ``transverse fundamental'' current
\begin{align*}
\td c (\a) \, = \, \int_M \a , \qquad \a \in \Om^0(M) .
\end{align*}
As in the previous case, taking ${\rho_0} = \phi, {\rho_1} = \phi^{-1}$ and using
the similar observation $\td c_0 (da) = 0$, one has
 \begin{align}  \label{pre-c0}
 \begin{split}
&2\Phi_{\rm C}(c_0)(a_0, a_1)= \td c_0 (da_1\, a_0 - da_0\, a_1) =
 \td c_0 \big(df_1\, {\rho_1}^*( f_0) - df_0 \,{\rho_0}^*( f_1)\big)  \\
 &= \int_M (f_0\circ {\phi}^{-1}) df_1  - \int_M (f_1\circ {\phi}) df_0 \\
 &= \int_M (f_0\circ {\phi}^{-1}) X_1(f_1) dx  - \int_M (f_1\circ {\phi}) X_1(f_0) dx  \\
 &= 2\int_M (f_0\circ {\phi}^{-1}) X_1(f_1) dx  =  2 \tau(X_1(a_1) a_0) 
 =  2 \tau(a_0 \s^{-1}X_1(a_1))  .
 \end{split}
 \end{align} 
Thus,
 \begin{align} \label{c0}
 \Phi_{\rm C}(c_0) (a_0, a_1) \, = \, \chi_\tau (\s^{-1}X_1) (a_0, a_1) .
 \end{align} 
Summing up, we have established the following result.

\begin{proposition} \label{expl-cocy}
Via the characteristic map \eqref{char-map},
the formulas \eqref{c0} and \eqref{c1} determine uniquely the 
cyclic cocycles $c^\Kc_0$ and $c^\Kc_1$ in $CC^1(\Kc_1,\, ^{\s^{-1}}\Cb)$,
\begin{align}
c^\Kc_0:= \one\ot\s^{-1} X_1, \qquad c^\Kc_1:= \one \ot \s^{-2}\s^1_{1,1} 
\end{align}
whose cohomology classes form a basis of $HP^\bullet (\Kc_1,\, ^{\s^{-1}}\Cb)$.
\end{proposition}
\medskip

We now recall the definition of the characteristic map for the action of $\Hc_1$
on
$\td\Ac = C_c^\infty (FM) \rtimes \Gb$. The group $\Gb$ acts on the frame bundle
$FM = \Rb \rtimes  \Rb^+$ by prolongation, \ie ,
 \begin{align} \label{frame-action}
 \phi (x, \ybo) = (\phi(x), \phi'(x) \ybo), \quad x\in \Rb , \ybo \in  \Rb^+, \quad
 \phi \in \Gb.
 \end{align}
Dual to the canonical framing by the horizontal and the vertical vector fields, 
$X_1 = \ybo \p_x$ , resp. $Y^1_1= \ybo  \p_y$, there is the
basis of $1$-forms $\theta^1 = \ybo^{-1} dx$, $\om^1_1 = \ybo^{-1} d \ybo$.
The volume form 
$\td \varpi = \theta^1 \wg \om^1_1 \wg = \ybo^{-2} \, dx \wg d \ybo$ 
is $\Gb$-invariant and gives rise to a canonical trace
 $\td\tau:\td\Ac\ra \Cb$,
\begin{align} \label{up-trace}
\begin{split}
&\td\tau(fU^\ast_\phi)= \left\{\begin{matrix}\displaystyle   \int_{{FM}} f
\td\varpi,& \text{if} \quad \phi=\Id \\ &&\\
0, &\text{otherwise .}
\end{matrix}\right.
\end{split}
\end{align} 
Besides the usual trace property
\begin{align}\label{true-trace}
\td\tau(ab)=\td\tau(ba) , \qquad \forall \, a, b \in \Ac ,
\end{align}
 $\td\tau$ is also $\ve$-invariant with respect to the action of $\Hc_1$,  
 \begin{align} \label{delta-invariant}
\td\tau(h(a))=\ve(h)\td\tau(a) , \qquad \forall \, h \in \Hc_n, \, a \in \td\Ac .
\end{align}
Thanks to these properties the
characteristic map $\chi_{\td\tau}$ from the standard Hopf cyclic $(b, B)$-complex
$CC^\bullet(\Hc_1; \,\Cb_\d)$ to the cyclic cohomology $(b, B)$-complex
$ CC^\bullet(\Ac)$, defined by
\begin{align} \label{up-char}
 & \chi_{\td\tau} (h^1, \ldots,h^q)(a_0\ldots,a_q )=\td\tau(a_0 h^1(a^1)\cdots h^q(a_q)), \\ \notag
 &\qquad h^1, \ldots,h^q \in \Hc_n , \qquad  a_0, \ldots, a_q \in \td\Ac ,
\end{align}
is a map of cyclic complexes. Moreover, this map is injective. 

The relative version of this map for the pair $(\Hc_1, \GL^+_1)$ 
(see~\cite{M-EC} for general dimension $n \in \Nb$) is 
obtained as follows. For a $q$-cochain in the relative cohomology complex
$ c = \sum_\a \one \ot \dot h_\a^1 \ot \cdots \ot \dot h_\a^q \in 
\Cb_\d \ot_{\Uc(\Fg \Fl_1)} {\Qc_1^{\cop}}^{\ot q}$, with $\dot h \in \Qc_1$ denoting the
class of $h \in \Hc_1$, one defines
\begin{align} \label{rel-char-map}
\chi_{\rm rel} (c) (a_0\ldots,a_q ) :=\sum_\a
\tau(\td a_0  h_\a^1(\td a^1)\cdots h_q^q(\td a_q) \mid_{\ybo = 1}) ;
\end{align}
here  $\td a = \td f U^*_{\td \phi} \in \td \Ac$ stands for the natural lift to the frame bundle
$FM = \Rb \rtimes  \Rb^+$  of $a = f U^*_\phi \in \Ac$. Note that the 
twisted trace $\tau$ of $\Ac$
is applied only after evaluating at $\ybo = 1$ a product which was 
performed in $\td \Ac$; as will be seen in the computation below, this compensates
for the twisting.

\begin{proposition} \label{expl-rel-cocy}
Via the characteristic map \eqref{rel-char-map},
the formulas \eqref{c0} and \eqref{c1} determine uniquely the 
cyclic cocycles $c^\Hc_0$ and $c^\Hc_1$ in 
$\Cb_\d \ot_{\Uc(\Fg \Fl_1)} {\Qc_1^{\cop}}$,
\begin{align}
c^\Hc_0:= \one\ot  \dot X_1, \qquad c^\Hc_1:= \one \ot \dot \d^1_{1,1} 
\end{align}
whose cohomology classes form a basis of $HP^\bullet (\Hc_1 , \GL^+_1\, ; \Cb_\d)$.
\end{proposition}

\begin{proof}
To recognize $c^\Hc_0$, we note that in $\td \Ac$ one has
\begin{align*}
\td f_0  U^*_{\td \phi} \cdot X_1(f_1 U_{\td \phi}) &= 
f_0 (x) U^*_{\td \phi} \cdot \ybo\, f'_1(x) U_{\td \phi}  =
f_0 (x) \, \td \phi^*(\ybo f'_1(x)) \\
&= f_0 (x)\, \phi'(x) \, \ybo\,  f'_1(\phi(x)) ,
\end{align*}
therefore,
\begin{align*}
\chi_{\rm rel} (\one\ot  \dot X_1) = \int_M  f_0 (x)\, \phi'(x) \, f'_1(\phi(x)) \, dx
 = \int_M  f_0(\phi^{-1} (x))\,  f'_1(x) \, dx .
 \end{align*}
By \eqref{pre-c0}, the last integral coincides with $ \Phi_{\rm C}(c_0)$. 
Since  $\Phi_{\rm C}$ and $\Phi^{\GL_1}_{\rm C}$ are the same, this proves
the first identity in the statement.

Similarly, 
\begin{align*}
\td f_0  U^*_{\td \phi} \cdot \d^1_{1,1} (f_1 U_{\td \phi}) &= 
f_0 (x) U^*_{\td \phi} \cdot \ybo\, \frac{(\phi^{-1})'' (x)}{(\phi^{-1})' (x)} \, f_1(x) U_{\td \phi}  \\
&= f_0 (x)\, \phi'(x) \, \ybo\, \frac{(\phi^{-1})''((\phi(x))}{(\phi^{-1})' ((\phi(x))} \, f_1((\phi(x)) ,
\end{align*}
which implies
\begin{align*}
\chi_{\rm rel} (\one\ot  \dot \d^1_{1,1}) &= 
\int_M  f_0 (x)\, \phi'(x) \, \frac{(\phi^{-1})''((\phi(x))}{(\phi^{-1})' ((\phi(x))} \, f_1((\phi(x)) \, dx \\
 &= \int_M  f_0(\phi^{-1} (x))\,  
 \frac{(\phi^{-1})''(x)}{(\phi^{-1})' (x)} \, f_1(x) \, dx .
 \end{align*}
By \eqref{pre-c1}the result  coincides with $ \Phi_{\rm C}(c_1)$, completing the proof. 
\end{proof}

\medskip

\begin{remark}In the model given by
the Hopf-Chevalley-Eilenberg bicomplex \eqref{wedge-coinv-FH}
the above classes are represented by the cocycles
\begin{align}\label{cocycles-first-H}
&C_0( \hat{R}|_0)=1\in \Fc_\Hc , \quad  
C_1( \hat{R}|_0)=\t^1\ot 1\wdg\a^1_{1,1}\in V^\ast\ot \wdg^2\Fc_\Hc .
\end{align}
The corresponding cocycles, via the map $\Theta_\Kc$, in
the bicomplex \eqref{wedge-coinv-FK} are given by
\begin{align}\label{cocycles-first-K}
&C_0^\dagger( \hat{R}|_0)=1\in \Fc_\Kc , \quad  C_1^\dagger( \hat{R}|_0)=\t^1\ot 1\wdg\b^{-1}\b^1_{1,1}\in V^\ast\ot \wdg^2\Fc_\Kc.
\end{align}
\end{remark}
\bigskip

In contrast to the case of $\Hc_1$, the Hopf cyclic cohomology of $\Kc_1$
contains the Chern class $c^\Kc_1$ but
is missing the Godbillon-Vey class. 
The reason is the algebraic nature of the
cochains of the latter. We proceed to show that
if one allows transcendental cocycles, the Chern class $c^\Kc_1$
vanishes, while the Godbillon-Vey class reappears.

Indeed, let $\tilde{\Kc}_1$ be the Hopf algebra obtained by adjoining a primitive element
$\log \s$ to $\Kc_1$, subject to the commutation relations
 \begin{align}\notag
 [X, \log \s] = \s^{-1}{\s^1_{1,1}},\quad [\log \s, \s_k] = 0 ,
  \;\;\forall k \in \Nb.
 \end{align}
 
 With the Hopf algebraic structure dictated  by the Leibniz rule as follows, 
 \begin{equation}
\D (\log \s) = \log \s \ot 1 + 1 \ot \log \s_1 ,\\
\end{equation} 

\begin{proposition}
The $2$-cochain
 \begin{align} \label{gv}
\uc_1 := \one\ot \s^{-1} X_1 \ot \s^{-2} \log \s 
\end{align}
 is a  Hochschild cocycle  whose Connes boundary is $c^\Kc_1$. 
 \end{proposition}

\begin{proof}
By transfer via the characteristic map $\chi_\tau$, we can work in the cyclic complex
of $\Ac_\G$.  Denoting the transported cochain by
 \begin{equation*}
\td\uc_1(f_0U^\ast_{\phi_0}, f_1U^\ast_{\phi_1}, f_2U^\ast_{\phi_2})= 
\tau\bigg(f_0U^\ast_{\phi_0} {\phi_1'}^{-1} \cdot {f_1'}U^\ast_{\phi_1} {\phi_2'}^{-1}\cdot { f_2\cdot \log(\phi_2')}U^\ast_{\phi_2}\bigg),
\end{equation*}
 let us first
check that it is a Hochschild cocycle. 
\begin{align*}
&b(\td\uc_1)(f_0U^\ast_{\phi_0}, \ldots, f_3U^\ast_{\phi_3})\\
&=\tau\bigg(f_0\cdot (f_1\circ\phi_0)U^\ast_{\phi_1\circ\phi_0} {\phi_1'}^{-1}\cdot {f_2'}U^\ast_{\phi_2} {\phi_3'}^{-1}\cdot { f_3\cdot \log(\phi_3')}U^\ast_{\phi_3}\bigg)\\
&- \tau\bigg(f_0U^\ast_{\phi_0} {(\phi_2\circ\phi_1)'}^{-1}{(f_1\cdot (f_2\circ \phi_1))'}U^\ast_{\phi_2\circ\phi_1} {\phi_3'}^{-1}\cdot { f_3\cdot \log(\phi_3')}U^\ast_{\phi_3}\bigg)\\
&+\tau\bigg(f_0U^\ast_{\phi_0} {\phi_1'}^{-1}\cdot {f_1'} U^\ast_{\phi_1}{(\phi_3\circ\phi_2)'}^{-1}  { {f_2\cdot( f_3\circ\phi_2)}\cdot \log({(\phi_3\circ\phi_2})')}
U^\ast_{{\phi_3\circ\phi_2}}\bigg)\\
&-\tau\bigg(f_3\cdot ( f_0\circ \phi_3 )U^\ast_{\phi_0\circ\phi_3} {\phi_1'}^{-1}\cdot {f_1'}U^\ast_{\phi_1} {\phi_2'}^{-1}\cdot{ f_2\cdot\log(\phi_2')} U^\ast_{\phi_2}\bigg).
\end{align*}

By the  Leibniz rule we see that, 
\[{(\phi_2\circ\phi_1)'}^{-1}\cdot {(f_1\cdot f_2\circ \phi_1)'}=
{(\phi_2'\circ\phi_1)\cdot \phi_1'}^{-1}\cdot {f_1'\cdot (f_2\circ \phi_1)}+{\phi_2'\circ\phi_1}^{-1}\cdot {f_1\cdot (f_2'\circ \phi_1)}\]

and 

\begin{align*}&{(\phi_3\circ\phi_2)'}^{-1}\cdot { {f_2\cdot( f_3\circ\phi_2)}\cdot \log({(\phi_3\circ\phi_2})')}\\
&={\phi_3'\circ\phi_2}^{-1}\cdot{\phi'_2}\;( f_3\cdot \log({\phi_3'))\circ\phi_2}+ {\phi_2'}^{-1}\cdot{f_2\cdot\log(\ph_2')}
\;\;{\phi_3'\circ\phi_2}^{-1}\cdot{ {f_3\circ\phi_2}}
\end{align*}

which yields that 

\begin{align*}
&b(\td\uc_1)(f_0U^\ast_{\phi_0}, \ldots, f_3U^\ast_{\phi_3})\\
&= \tau\bigg( f_0U^\ast_{\phi_0} {\phi_1'} ^{-1}\cdot{f_1'}U^\ast_{\phi_1} {\phi_2'}^{-1}\cdot { f_2\cdot\log(\phi_2')}U^\ast_{\phi_2} 
{\phi_3'}^{-1}\cdot{f_3}U^\ast_{\phi_3}\bigg)- \\
& \tau\bigg(f_3U^\ast_{\phi_3} f_0U^\ast_{\phi_0} {\phi_1'}^{-1}{f_1'} U^\ast_{\phi_1} {\phi_2'}^{-1}{ f_2\cdot\log(\phi_2')} U^\ast_{\phi_2}\bigg).
\end{align*}

Finally by the  $\s^{-1}$-tracial property   of $\tau$ we see $b(\td\uc_1)=0$.  

To show that $B(\td\uc_1)=c^\Kc_1$, 
we first observe that $\td\uc_1$ is normalized. So we continue  by 
\[B(\td\uc_1)(a,b)=\td\uc_1(1,a,b)-\td\uc_1(1,b,a). \]
\begin{align*}
&B(\td\uc_1)(f_0U^\ast_{\phi_0}, f_1U^\ast_{\ph_1})\\
&=\tau\bigg( {\phi_0'}^{-1}\cdot{f_0'} U^\ast_{\phi_0} {\phi_1'}^{-1}\cdot { f_1\cdot \log(\phi_1')}U^\ast_{\phi_1}\bigg)
-\tau\bigg( {\phi_1'}^{-1}\cdot {f_1'}U^\ast_{\phi_1} {\phi_0'}^{-1}\cdot{ f_0\cdot \log(\phi_0')} U^\ast_{\phi_0}\bigg).
\end{align*}
Without loss of generality we assume that $\phi_0^{-1}=\phi_1$. Then again by using the $\s^{-1}$-tracial property of $\tau$ we have

\begin{align*}
&\tau\bigg( {\phi_1'}^{-1} \cdot{f_1'}U^\ast_{\phi_1} {\phi_0'}^{-1}\cdot { f_0\cdot \log(\phi_0')}U^\ast_{\phi_0}\bigg)=
\tau\bigg(f_0 U^\ast_{\phi_0} {\phi_0'\circ\phi_1}^{-1}\cdot{  \log(\phi_0'\circ\phi_1)} \cdot{(\phi_1')^{-2}}{f_1'} U^\ast_{\phi_1} \bigg)\\
&=-\tau\bigg(f_0 U^\ast_{\phi_0}    {\phi_1'}^{-2} \cdot {f_1' \cdot \log(\phi_1')}U^\ast_{\phi_1} \bigg)
\end{align*}

On the other hand one uses the integration by part property of $\tau$ to see
\begin{align*}
&\tau\bigg( {\phi_0'}\cdot {f_0'}U^\ast_{\phi_0} {\phi_1'}^{-1} { f_1\cdot \log(\phi_1')}U^\ast_{\phi_1}\bigg)=-\tau\bigg(  f_0  U^\ast_{\phi_0}{\phi_1'}^{-1}\left( \phi_1' \cdot{\phi_1'}^{-1}\cdot { f_1\cdot \log(\phi_1')} \right)'U^\ast_{\phi_1}\bigg)\\
&=-\tau\bigg(  f_0  U^\ast_{\phi_0}  {\phi_1'}^{-1}\cdot{f_1'\cdot \log(\phi_1')}U^\ast_{\phi_1}\bigg)-\tau\bigg(  f_0  U^\ast_{\phi_0} 
{(\phi_1')^{-2}}\cdot{f_1\cdot\phi_1''} U^\ast_{\phi_1}\bigg).
\end{align*}
This completes the proof of the claimed result.
\end{proof}
\medskip

When dealing with $\tilde{\Kc}_1$, 
one may import the  Godbillon-Vey  cocycle from the Bott bicomplex via
the characteristic map 
$$\chi_\tau :\;\;^{\s^{-1}}\Cb\ot {\td\Kc}^{\ot 2}\ra CC^2(\Ac) \, ,
$$
as we show below.

\begin{proposition} \label{gv}
The element
 \begin{align*} 
{\rm GV} :=\one\ot  \log \s \ot \s^{-2} {\s^1_{1,1}} - \one\ot\s^{-2} {\s^1_{1,1}} \ot \s^{-1} \log \s \in
\;\;^{\s^{-1}}\Cb\ot \tilde{\Kc}_1 \ot \tilde{\Kc}_1
\end{align*}
is a cyclic cocycle.
\end{proposition}

\begin{proof}
On the one hand, one has
 \begin{align} \label{b-gv1}
 \begin{split}
 b(\one\ot\log \s &\ot \s^{-2} {\s^1_{1,1}}) = \one\ot1 \ot \log \s \ot \s^{-2} {\s^1_{1,1}} \\
 &- \one\ot1 \ot \log \s \ot \s^{-2} {\s^1_{1,1}}
- \one\ot\log \s \ot 1 \ot \s^{-2} {\s^1_{1,1}}) \\
& +  \one\ot\log \s \ot 1 \ot \s^{-2} {\s^1_{1,1}} +
 \one\ot\log \s \ot \s^{-2} {\s^1_{1,1}} \ot \s^{-1}\\
 & - \one\ot\log \s \ot \s^{-2} {\s^1_{1,1}} \ot \s^{-1} = 0 ,
 \end{split}
\end{align}
and also
 \begin{align} \label{b-gv2}
 \begin{split}
 b(\one\ot\s^{-2} {\s^1_{1,1}} &\ot \s^{-1} \log \s) =\one\ot 1 \ot \s^{-2} {\s^1_{1,1}} \ot \s^{-1} \log \s\\
 &-\one\ot \s^{-2} {\s^1_{1,1}} \ot \s^{-1} \ot \s^{-1} \log \s -  \one\ot1 \ot \s^{-2} {\s^1_{1,1}} \ot \s^{-1} \log \s\\
& + \one\ot \s^{-2} {\s^1_{1,1}} \ot \s^{-1} \log \s  \ot \s^{-1} + \one\ot\s^{-2} {\s^1_{1,1}} \ot \s^{-1} \ot  \log \s \\
& -  \one\ot\s^{-2} {\s^1_{1,1}} \ot \s^{-1} \log \s  \ot \s^{-1} \, = \, 0 .
  \end{split}
\end{align}

On the other hand,
 \begin{align} \label{t-gv1}
 \begin{split}
 \tau_2(\one\ot \log \s \ot \s^{-2} {\s^1_{1,1}}) &= \one\ot \D S( \log \s) \cdot \s^{-2} {\s^1_{1,1}}  \ot \s^{-1} \\
 &= - \one\ot\D(\log \s) \cdot \s^{-2} {\s^1_{1,1}} \ot \s^{-1}\\
&= - (\one\ot\s^{-2} {\s^1_{1,1}} \log \s  \ot \s^{-1} + \one\ot\s^{-2} {\s^1_{1,1}} \ot \s^{-1} \log \s ) ,
 \end{split}
\end{align}
while, taking into account that $\quad  S(\s^{-2} {\s^1_{1,1}})  =  - \s^{-1} {\s^1_{1,1}}$,
 \begin{align} \label{t-gv2}
 \begin{split}
\tau_2 (\one\ot\s^{-2} {\s^1_{1,1}} &\ot \s^{-1} \log \s) = \one\ot\D S(\s^{-2} {\s^1_{1,1}}) \cdot
 \s^{-1} \log \s  \ot \s^{-1}\\
 & = -\one\ot\D  (\s^{-1} {\s^1_{1,1}}) \cdot  \s^{-1} \log \s  \ot \s^{-1}\\
 &= - (\one\ot\s^{-2} {\s^1_{1,1}} \log \s  \ot \s^{-1} +\one\ot \log \s \ot \s^{-2} {\s^1_{1,1}}) .
 \end{split}
 \end{align}
 Therefore,
 \begin{align}
  \begin{split}
  \tau_2(\one\ot\log \s &\ot \s^{-2} {\s^1_{1,1}} - \one\ot\s^{-2} {\s^1_{1,1}} \ot \s^{-1} \log \s) \\
  &=  \one\ot \log \s \ot \s^{-2} {\s^1_{1,1}} -\one\ot \s^{-2} {\s^1_{1,1}} \ot \s^{-1} \log \s .
  \end{split}
  \end{align}
\end{proof}

\end{document}